\documentclass[a4paper]{scrartcl}
\usepackage[utf8]{inputenc}
\usepackage[british]{babel}
\usepackage[T1]{fontenc}
\usepackage{lmodern}
\usepackage{microtype}
\usepackage{amsfonts,amsthm,amssymb,amsmath}
\usepackage{thm-restate}
\usepackage{mathtools}
\usepackage{letterswitharrows}
\usepackage{xcolor}
\definecolor{skyblue}{HTML}{3465a4}
\usepackage{graphicx}
\usepackage{tikz}
\usetikzlibrary{arrows}
\usepackage[verbose]{wrapfig}
\usepackage[colorlinks,unicode,bookmarks,pdfusetitle,linkcolor={red!60!black},citecolor={green!60!black},urlcolor={blue!60!black}]{hyperref}
\usepackage[capitalize,nameinlink]{cleveref}
\crefname{enumi}{}{}
\crefname{property}{property}{properties}
\crefname{LEM}{Lemma}{the Lemmas}
\usepackage[shortlabels]{enumitem}
\usepackage{thmtools, thm-restate} %For duplicating theorem numbers.
\usepackage{cite}
\usepackage[abbrev]{amsrefs}
\usepackage{doi}
\usepackage{appendix}
\renewcommand{\PrintDOI}[1]{\doi{#1}}

\newtheorem{THM}{Theorem}[section]
\newtheorem{LEM}[THM]{Lemma}
\newtheorem{COR}[THM]{Corollary}

\theoremstyle{definition}
\newtheorem{EX}[THM]{Example}

\newcommand{\abs}[1]{\lvert#1\rvert}
\newcommand{\menge}[1]{\left\{#1\right\}}

\renewcommand{\phi}{\varphi}

\newcommand{\Nbb}{\mathbb{N}}

\newcommand{\N}{\mathbb{N}}

\newcommand{\join}{\lor}
\newcommand{\meet}{\land}
\newcommand{\sub}{\subseteq}
\newcommand{\sm}{\smallsetminus}
\newcommand{\es}{\emptyset}
\renewcommand{\le}{\leqslant}\renewcommand{\leq}{\leqslant}
\renewcommand{\ge}{\geqslant}\renewcommand{\geq}{\geqslant}

\newcommand{\cC}{\mathcal{C}}
\newcommand{\cD}{\mathcal{D}}

\newcommand{\cF}{\mathcal{F}}
\newcommand{\cG}{\mathcal{G}}

\newcommand{\cM}{\mathcal{M}}

\newcommand{\cP}{\mathcal{P}}

\def\restricts{\mathbin\restriction}

\newcommand{\shifting}[2]{\mathop{f\!\downarrow}\nolimits^{#1}_{#2}}

\renewcommand{\max}{\operatorname{\textsf{max}}}

\title{Obtaining trees of tangles from tangle-tree duality}
\author{Christian Elbracht\and Jakob Kneip\and Maximilian Teegen}
\date{19th November 2020}
\begin{document}
\maketitle
\begin{abstract}
    We demonstrate the versatility of the tangle-tree duality theorem for abstract separation systems \cite{TangleTreeAbstract}
    by using it to prove tree-of-tangles theorems.
    This approach allows us to strengthen some of the existing tree-of-tangles theorems by bounding the node degrees in them.
    We also present a slight strengthening and simplified proof of the duality theorem, which allows us to derive a tree-of-tangles theorem also for tangles of different orders.
\end{abstract}
\section{Introduction}
Over the last twenty years, tangles, originally developed by Robertson and Seymour as a tool in their monumental graph minor project \cite{GMX}, have evolved a lot. Originally defined specifically just for graphs, they have since been generalized not only to other combinatorial structures like matroids \cite{TanglesInMatroids} but even into an abstract setting in which concrete separations are replaced by an abstract poset with just some simple properties that reflect those that separations typically have \cites{AbstractSepSys,ProfilesNew,TangleTreeAbstract}. In all these settings, the general idea of tangles is to use them as a method to indirectly capture highly cohesive substructures of various kinds, by deciding for every low-order separation on which side of that separation the desired structure lies.

Already in the original work by Robertson and Seymour the theory of tangles has two major theorems: the tree-of-tangles theorem and the tangle-tree duality theorem. These two form the main pillars of tangle theory, and thereby of a central aspect of graph minor theory.

The first of these theorems allows one to distinguish all the tangles in a tree-like way, displaying their relative position in the underlying combinatorial structure. One of the most abstract variants of the tree-of-tangles theorem reads as follows:
\begin{THM}[name={\cite{AbstractTangles}*{Theorem 6}}, restate=totDaniel]\label{thm:Daniel} 
 Let $\vS$ be a structurally submodular separation system and $\cP$ a set of profiles of $S$. Then $\vS$ contains a tree set that distinguishes $\cP$.
\end{THM}
\noindent Profiles are the most general class of objects that one can think of as tangles.

The tangle-tree duality theorem, on the other hand, provides a tree-like dual object to tangles which, if no tangle exists, serves as a witness that there can be no tangle.
In this paper we demonstrate the versatility of the most abstract version of this duality theorem: we deduce \cref{thm:Daniel} and some of its variations from the tangle-tree duality theorem, reducing the two pillars of abstract tangle theory to a single pillar.

In order to use tangle-tree duality to deduce tree-of-tangles theorems like \cref{thm:Daniel},  we exploit the generality of the most abstract version of the tangle-tree duality theorem, which reads as follows:
\begin{restatable}[Tangle-tree duality theorem~\cite{TangleTreeAbstract}*{Theorem 4.3}]{THM}{TTD}\label{thm:duality}
	Let $ U $ be a universe containing a finite separation system $ S\sub U $ and let $ \cF\sub 2^{\vU} $ be a set of stars such that $ \cF $ is standard for $ \vS $ and $ \vS $ is $ \cF $-separable. Then exactly one of the following statements holds:
	\begin{itemize}
		\item there is an $ \cF $-tangle of $ S $;
        \item there is an $ S $-tree over $ \cF $.
	\end{itemize}
\end{restatable}

The strength of~\cref{thm:duality} lies in the flexibility it allows in the choice of~$ \cF $. This set~$ \cF $ can be tailored to capture a wide variety of tangles and clusters, allowing~\cref{thm:duality} to be employed in a multitude of different settings (\cite{TangleTreeGraphsMatroids,AbstractTangles}). The freedom in choosing and manipulating~$ \cF $ will also allow us to achieve our goal of deducing tree-of-tangles theorems from \cref{thm:duality}: by a clever choice of~$ \cF $ we can ensure that there is no~$ \cF $-tangle of~$ S $, and that the~$ S $-tree over~$ \cF $ one then obtains will be a tree of tangles. We present multiple variations of this idea throughout this paper.

In terms of simplicity and brevity, reducing the tree-of-tangles theorem to the tangle-tree duality theorem in this way cannot compete with its direct proofs in \cite{ProfilesNew},\cite{AbstractTangles} or \cite{FiniteSplinters}, our general purpose solution to obtaining tree-of-tangles theorems in a wide range of structures. (There, we showed an even more general theorem than \cref{thm:Daniel} which no longer mentions tangles or profiles at all, but just talks about sets of separations fulfilling one simple-to-check condition.)

Instead of competing in terms of simplicity and brevity just for a proof of the tree-of-tangles theorem, the aim of this paper is to bridge the two parts of the theory needed for their classical proofs.
This can be viewed in two ways. Firstly, that we introduce tools from tangle-tree duality into the world of trees of tangles, which gives us a new method for building trees in this context very unlike the proofs in \cites{ProfilesNew,AbstractTangles,FiniteSplinters}.

Secondly, and perhaps more importantly, from the perspective of tangle-tree duality this may be viewed as introducing a new range of ways of how to apply the duality theorem by a careful choice of~$\cF$.
Previous applications of \cref{thm:duality} all worked with largely similar choices of $\cF$, all designed to capture some notion of `width', whereas we specifically construct $\cF$ in such a way that no $\cF$-tangle can exist, thereby making sure that \cref{thm:duality} gives us the dual object which will be the desired tree-of-tangles.

A new result that we get from this method is that it allows us to bound the degrees of the nodes in a tree of tangles in some contexts. Getting such a degree condition out of the original proofs does not appear to be simple.

To achieve our last result, we prove a strengthened version of \cref{thm:duality}, which we present in \cref{subsec:duality} along with a simpler proof than the original one in \cite{TangleTreeAbstract}.

The structure of this paper is as follows. In \cref{sec:defs} we will repeat the required definitions from \cite{ProfilesNew,AbstractTangles,AbstractSepSys,TangleTreeAbstract,TangleTreeGraphsMatroids}. In ~\cref{sec:finite_merging_submodular} we prove our first basic tree-of-tangles theorem, for structurally submodular separation systems. A refined version of this argument will be given in~\cref{sec:finite_merging_degrees}, where we show that the approach via tangle-tree duality yields a bound on the degrees of the nodes in a tree of tangles. In \cref{sec:efficiency} we present a more involved argument to obtain a tree of tangles that  distinguishes a set of profiles \emph{efficiently}. Again, this approach can be used to obtain a result about the degrees in such a tree, and we do so in \cref{sec:efficient_degrees}. 
In \cref{sec:different_order} we prove a tree-of-tangles theorem for tangles of different orders. For that we need our stronger version of the tangle-tree duality theorem, which we state in \cref{subsec:duality}. The proof of this stronger duality theorem in \cref{subsec:duality} also offers a new, and maybe simpler, proof of the original tangle-tree duality theorem \cref{thm:duality}.
In our final section, \cref{subsec:different_order}, we then use this stronger tangle-tree duality theorem to obtain a tree-of-tangles theorem for profiles of different order.

\section{Terminology and background}\label{sec:defs}

Since we combine the theory of tangle-tree duality and of trees of tangles we need the terminology of both. Consequently, this results in a large number of definitions which need to be understood for the comprehension of this paper.
We employ the frameworks of \cites{ProfilesNew,AbstractTangles,AbstractSepSys,TangleTreeAbstract,TangleTreeGraphsMatroids}. 
For reference, we offer a recap of the definitions that we will use, split up according to their context.
The cited sources provide more in-depth motivation of the respective set-ups.

\subsection{Basic definitions of abstract separation systems\texorpdfstring{ (see \cite{AbstractSepSys})}{}}

A separation system $S=(\vS,\leq,^*)$ consists of a finite poset $\vS$ together with an involution $^*$ which is order-reversing, i.e., $\vs \leq \vt \Leftrightarrow \vs^* \geq \vt^*$. We call the elements of $\vS$ \emph{(oriented) separations} and denote, given an oriented separation $\vs\in \vS$, the image of $\vs$ under $^*$ as the \emph{inverse} of $\vs$.

The pair $\{\vs,\sv\}$ of $\vs$ together with its inverse $\sv$ is denoted as $s$ and called the \emph{underlying unoriented separation} of $\vs$. Given $s$, we say that $\vs$ and $\sv$ are the \emph{orientations} of $s$. The set of all the underlying unoriented separations for a set $\vT\subseteq \vS$ is denoted as $T$, so $S$ is the set of all unoriented separations of separations in $\vS$. Conversely, when given a set $T$ of unoriented separations, we denote as $\vT$ the set of all orientations of separations in $T$. For brevity, we mean by the term `separations' both oriented and unoriented separations if the intended meaning is clear from the context.

 We say that $\vs\in \vS$ is \emph{small} if it is less then its inverse, that is if $\vs\le\sv$. If there is an unoriented separation $r\neq s$ such that $\vs\le\vr$ and $\vs\le \rv$, then $\vs$ is called \emph{trivial}. Note that every trivial separation is small, since $\vs\le\rv$ implies that $\sv\ge\vr$ and thus $\vs\le\vr\le\sv$.
 
\pagebreak[4]
The inverse $\sv$ of a small separation $\vs$ is called \emph{cosmall} and likewise the inverse of a trivial separation is called \emph{cotrivial}.
A set of oriented separations is \emph{regular} if it contains no cosmall separation. 

We say that a separation $s$ (and likewise its orientations) is \emph{degenerate} if $\vs=\sv$.

We say that two unoriented separation $r$ and $s$ from $S$ are \emph{nested} if they have orientations $\vr,\vs$ such that $ \vr\le\vs$. Two separations $r$ and $s$ \emph{cross} if they are not nested. Two oriented separations $\vr$ and $\vs$ cross or are nested if the underlying unoriented separations $r$ and $s$ cross or are nested, respectively. Note than in particular $\vr$ and $\vs$ can be nested even if they are incomparable, for instance if $\vr\le\sv$.
A set of separations is called \emph{nested} if its elements are pairwise nested.

A nested set $T$ of unoriented separations is called a \emph{tree set} in a separation system $S$, if $T$ does not contain any separation $s$ which has a trivial orientation in $\vT$.
A tree-set $T$ is \emph{regular} if $\vT$ is regular.

A universe $U=(\vU,\leq,^*,\join,\meet)$ is a separation system $(\vU,\leq,^*)$ together with join and meet operators $\join,\meet$ which turn the poset $(\vU,\leq)$ into a lattice.
For universes DeMorgan's law holds: \[
    (\vs \join \vt)^* = \sv \meet \tv
\]
Given two unoriented separations $s$ and $t$ in $U$, we call the unoriented separations corresponding to $\vs\join \vt,\vs\join\tv,\sv\join\vt$ and $\sv\join\tv$ the \emph{corner separations}, or \emph{corners} for short, of $s$ and $t$.

One often-used property of universes is the so-called \emph{fish lemma}:
\begin{LEM}[\cite{AbstractSepSys}*{Lemma 3.2}]\label{lem:fish}
Let $r,s\in U$ be two crossing separations. Every separation $t$ that is nested with both $r$ and $s$ is also nested with all four corner separations of $r$ and~$s$.
\end{LEM}

Given a separation system $S$, a subset $O\subseteq \vS$ is \emph{antisymmetric} if $|O\cap \{\vs,\sv\}|\le 1$ for every $s\in S$.

An \emph{orientation} of $S$ is an antisymmetric subset $O\subseteq \vS$ such that $\vs\in O$ or $\sv\in O$ for every $s\in S$. Such an orientation $O$ is \emph{consistent} if $O$ does not contain any $\vr$ and $\vs$ such that $ \rv\le\vs $ and $ r\ne s $.

Some universes $U$ come with an \emph{order function}, a function $\abs{\cdot}\colon\vU\to \Nbb_0$ which is invariant under $^*$, that is, for any $s\in U$, we have ${\abs{\vs}=\abs{\sv}\eqqcolon\abs{s}}$.
Such an order function is called \emph{submodular} if, for all $\vs,\vt\in \vU$, \[ \abs{\vs}+\abs{\vt}\ge \abs{\vs\join\vt}+\abs{\vs\meet\vt}\,. \]

A universe $\vU$ together with such a submodular order function is called a \emph{submodular universe}. Given a submodular universe we denote as $\vS_k\subseteq \vU$, for $k\in\N$, the separation system consisting of all separations $\vs\in \vU$ satisfying $\abs{\vs}<k$.

We say that a separation system $\vS$ inside a universe is \emph{structurally submodular} (some literature omits the `structurally') if, for all $\vs,\vt\in \vS$, at least one of $\vs\join \vt$ and $\vs\meet\vt$ also lies in $\vS$.
Note that, if $\vU$ is a submodular universe, then every $\vS_k\subseteq \vU$ is structurally submodular.

\subsection{The tree-of-tangles theorem\texorpdfstring{ (see \cite{ProfilesNew})}{}}
A separation $s$ is said to \emph{distinguish} two orientations $ O_1 $ and $ O_2 $ of two, possibly distinct, separation systems inside $U$, if $s$ has an orientation $\vs$ such that $ \vs\in O_1 $ and $ \sv\in O_2 $.
If $U$ comes with an order function we say that such an $s$ distinguishes $O_1$ and $O_2$ \emph{efficiently} if there is no $r$ with $|r|<|s|$ which distinguishes $O_1$ and $O_2$.

A consistent orientation $O$ of a separation system $S\subseteq U$ inside some universe $U$ is said to be a \emph{profile} if
it satisfies the profile property:
\[ \forall \,\vr,\vs\in P \colon (\rv\meet\sv)\notin P \tag{P}\label[property]{property:P} \]

For a universe $U$ with an order function, a \emph{$k$-profile in $\vU$} is a profile of $S_k \subseteq U$. We say that $P$ is a \emph{profile in $\vU$} if $P$ is a $k$-profile in $\vU$ for some $k$. If $P$ is a $k$-profile in $\vU$, then $k$ is the \emph{order} of $P$.

Such a profile $P$ is \emph{robust} if moreover:
\[\forall \vs\in P,\vt\in \vU:\text{ if }|\vs\join \vt|<|\vs|\text{ and }|\vs\join\tv|<|\vs|, \text{ then either }\vs\join\vt\in P\text{ or }\vs\join\tv\in P\]

The tree-of-tangles theorem for $k$-profiles states the following:

\begin{THM}[\cite{ProfilesNew}*{Corollary 3.7}, modified]\label{thm:tot_hundermark}
Let $(\vU,\le,^*,\join,\meet,\abs{\ })$  be a submodular universe of separations. For every set $\cP$ of pairwise distinguishable robust regular profiles in $\vU$ there is a regular tree set $T = T(\cP)\subseteq \vU$ of separations such that:
\begin{enumerate}
 \item every two profiles in $\cP$ are efficiently distinguished by some separation in $T$;
\item every separation in $T$ efficiently distinguishes a pair of profiles in $\cP$.
\end{enumerate}
\end{THM}

Note that the original statement \cite{ProfilesNew}*{Corollary 3.7} included a third property which guaranteed that the resulting set $T$ is invariant under automorphisms. Our methods in this paper will not allow us to guarantee this, that is why we exclude this property from our version of \cite{ProfilesNew}*{Corollary 3.7}. For more discussion of this property, \emph{canonicity}, see \cites{ProfilesNew,CanonicalToT}.

Similarly, we have the tree-of-tangles theorem already mentioned in the introduction for structurally submodular separation systems which do not necessarily come in the form of an $S_k\subseteq U$:
\totDaniel*

\subsection{Tangle-tree duality\texorpdfstring{ (see \cite{TangleTreeAbstract})}{}}

Given some set $\cF$ of subsets of $S$, an \emph{$\cF$-tangle} of $S$ is a consistent orientation of $S$ which includes no subset in $\cF$.
Given a submodular universe $\vU$, we say that $\tau$ is an \emph{$\cF$-tangle in $\vU$} if $\tau$ is an $\cF$-tangle of some $S_k$. 
Observe that profiles are $\cP$-tangles for the set $\cP$ of all `profile triples' $\menge{\vr, \vs, (\vr \join \vs)^*} \subseteq~\vS$.

Often we will consider sets $\cF$ of \emph{stars}: A \emph{star} in $\vS$ is
a set $\sigma \subseteq \vS$ such that $\vs\le \tv$ for all $\vs,\vt\in \sigma$.

We say that a set $\cF$ \emph{forces} a separation $\vs\in \vS$ if $\{\sv\}\in \cF$.

$\cF$ is \emph{standard for $\vS$} if it forces all trivial separations, that is $\cF$ contains all singletons $\{\vs\}$ for cotrivial $\vs\in\vS$.

Given a tree $T$ we denote as $\vE(T)$ the set of orientations of edges of $T$.
This set is equipped with a natural partial order where $\ve \le \vf$ if and only if the unique path in $T$ from the tail of $\ve$ to the head of $\vf$ contains both the head of $\ve$ and the tail of $\vf$. This partial order, together with ${}^*$ the reversal of directed edges, turns $\vE(T)$ into a separation system.

Given a separation system $\vS$, an \emph{$S$-tree} $(T,\alpha)$ is a tree $T$ together with a function $\alpha\colon \vE(T) \to \vS$ which commutes with~$^*$, i.e., $\alpha(\ve) = \alpha(\ev)^*$.
The $S$-tree is \emph{order-respecting} if $\alpha$ preserves the partial order from $\vE(T)$, i.e., $\alpha(\ve) \le \alpha(\vf)$ whenever $\ve \le \vf$.
For $t \in V(t)$ we denote as $\alpha(t)$ the set $ \{\alpha(st) \mid s \in N(t)\}$.
Given some set $\cF$ of subsets of $S$, an $S$-tree $(T,\alpha)$ is \emph{over $\cF$} if $\alpha(t) \in \cF$ for all $t \in V(T)$.

An $S$-tree $(T,\alpha)$ is \emph{irredundant}, if for any node $t\in V(T)$ and distinct neighbours $t',t''\in N(t)$ we have that $\alpha(t',t)\neq \alpha(t'',t)$.

Note that, if $\cF$ is a set of stars then any irredundant $S$-tree over $\cF$ is order-respecting.

Given a separation system $\vS$ inside a universe $\vU$ and $\vr,\vs_0 \in \vS$ with $\vs_0 \ge \vr$ and where $\vr$ is nondegenerate and notrivial in $\vS$, the \emph{shifting map} $\shifting{\vr}{\vs_0}$ is defined by letting, for every $\vs \ge \vr$, \[
    \shifting{\vr}{\vs_0}(\vs) = \vs \join \vs_0 \quad\text{and}\quad \shifting{\vr}{\vs_0}(\sv) = (\vs \join \vs_0)^*.
\]
This map is defined on $\vS_{\ge\vr}\sm \{\rv\}$, where $S_{\ge\vr}$ is the set of all separations $t \in S$ which have an orientation $\vt$ with $\vt \ge \vr$, and $\vS_{\ge\vr}$ is the set of all orientations of separations in~$S_{\ge\vr}$.

For an irredundant $S$-tree $(T,\alpha)$ over some set of stars with $\{\rv\} = \alpha(x)$, for some leaf $x$ of $T$, we write \[
    \alpha_{x,\vs_0} \coloneqq \shifting{\vr}{\vs_0} \circ\; \alpha\,.
\]
The resulting new tree $(T,\alpha_{x,\vs_0})$
is called the \emph{shift of $(T,\alpha)$ from $\vr$ to $\vs_0$} if the leaf $x$ is the only one which has $\alpha(x) = \{\rv\}$.

Given a separation system $\vS$ inside a universe $\vU$ and a star $\sigma \subseteq \vS$ a \emph{shift of $\sigma$ (to some $\vs_0 \in \vS$)} is a star of the form \[
    \sigma^{\vs_0}_\vx \coloneqq \menge{ \vx \join \vs_0 } \cup \menge{ \vy \meet \sv_0 \mid \vy \in \sigma \sm \menge{ \vx }},
\] where $\vx \in \sigma$.
Note that if, for some $\vr\in\vS$, we have $\vx \geq \vr$ then $\sigma^{\vs_0}_\vx$ is the image of $\sigma$ under $\shifting{\vr}{\vs_0}$.

A separation $\vs$ \emph{emulates $\vr$ in $\vS$} if $\vs \ge \vr$ and for every $\vt \in \vS \sm \menge{\rv}$ with $\vt \ge \vr$ we have $\vs \join \vr \in \vS$.
The separation $\vs$ emulates $\vt$ in $\vS$ \emph{for $\cF$} if additionally for every star $\sigma \in \cF$ with $\rv \notin \sigma$ and every $\vx \in \sigma$ with $\vx \ge \vr$ we have $\sigma_{\vx}^{\vs} \in \cF$.

Note that for an irredundant $S$-tree $(T,\alpha)$ over some set of stars $\cF$ with $\{\rv\} = \alpha(x)$, for some leaf $x$ of $T$, the shift from $\vr$ to $\vs_0$ is again an $S$-tree over $\cF$ if $\vs_0$ emulates $\vr$ in $\vS$ for $\cF$.

A separations system $S$ is \emph{separable} if for any two nontrivial nondegenerate separations $\vr_1, \vr_2 \in \vS$ with $\vr_1 \le \rv_2$ there exists a separation $\vs_0\in\vS$, with $\vr_1 \le \vs_0 \le \rv_2$ such that $\vs_0$ emulates $\vr_1$ in $S$ and $\sv_0$ emulates $\vr_2$ in $S$.
The separation system $S$ is \emph{$\cF$-separable} if we can choose, for any two such $\vr_1$ and $\vr_2$ which are nontrivial nondegenerate and not forced by $\cF$, such an $\vs_0$ so that $\vs_0$ emulates $\vr_1$ in $S$ for $\cF$ and $\sv_0$ emulates $\vr_2$ in $S$ for $\cF$.

The abstract tangle-tree duality theorem now states the following:

\TTD*

If, in the following, we speak of \emph{the} duality theorem, we mean~\cref{thm:duality}.

The condition of $\cF$-separability is sometimes split into two parts which, in sum, are stronger:
Firstly, that $S$ is separable and secondly that $\cF$ is \emph{closed under shifting}, that is, every shift $\sigma'$ of a star $\sigma \in\cF$ is also in $\cF$ if $\sigma' \subseteq \vS$.
(Compare~\cite{AbstractTangles}*{Lemma 12}.)

We shall need the following additional lemmas from the literature:

\begin{LEM}[{\cite{TangleTreeAbstract}*{Lemma 2.1}}]\label{lem:TreeSets_6.3}
    Every irredundant $S$-tree $(T,\alpha)$ over stars is order-respecting.
    In particular, $\alpha(\vE(T))$ is a nested set of separations in $\vS$.
\end{LEM}
\begin{LEM}[{\cite{TangleTreeAbstract}*{Lemma 2.2}}]\label{lem:TreeSets_6.4}
    Let $(T,\alpha)$ be an irredundant $S$-tree over a set $\cF$ of stars.
    Let $e$, $f$ be distinct edges of $T$ with orientations $\ve < \vf$ such that $\alpha(\ve) = \alpha(\fv) \eqqcolon \vr$.
    Then $\vr$ is trivial.

    In particular, $T$ cannot have distinct leaves associated with the same star $\menge{\rv}$ unless $\rv$ is trivial.
\end{LEM}
\begin{LEM}[{\cite{TangleTreeAbstract}*{Lemma 2.3}}]
\label{lem:TreeSets_6.2}
    If $(T,\alpha)$ is an $S$-tree over $\cF$, possibly redundant, then $T$ has a subtree $T'$ such that $(T',\alpha')$ is an irredundant $S$-tree over $\cF$, where $\alpha'$ is the restriction of $\alpha$ to $\vE(T')$.
    If $(T,\alpha)$ is rooted at a leaf $x$ and $T$ has an edge, then $T'$ can be chosen so as to contain $x$ and $e_x$, the edge incident to $x$ in $T$.
\end{LEM}
\begin{LEM}[{\cite{TangleTreeAbstract}*{Lemma 2.4}}]
\label{lem:TreeSets_6.5}
    Let $(T,\alpha)$ be an $S$-tree over a set $\cF$ of stars, rooted at a leaf $x$.
    Assume that $T$ has an edge, and that $\vr = \alpha(\ve_x)$ is nontrivial.
    Then $T$ has a minor $T'$ containing $x$ and $e_x$ such that $(T', \alpha')$, where $\alpha' = \alpha \restricts \vE(T')$, is a tight and irredundant $S$-tree over $\cF$.

    For every such $(T',\alpha')$ the edge $\ve_x$ is the only edge $\ve \in \vE(T')$ with $\alpha(\ve) = \vr$.
\end{LEM}

\begin{LEM}[\cite{AbstractTangles}*{Lemma 13}]\label{lem:separability}
Let $\vU$ be a universe of separations and $\vS\subseteq \vU$ a structurally submodular separation system. Then $\vS$ is separable.
\end{LEM}
Moreover, we shall need a variant of \cite{TangleTreeAbstract}*{Lemma 4.2} which follows with the exact same proof:
\begin{LEM}[\cite{TangleTreeAbstract}]\label{lem:shift_S-tree}
	Let~$ \cF\sub 2^\vU $ be a set of stars. Let~$ (T,\alpha) $ be a tight and irredundant~$ S $-tree with at least one edge, over some set of stars, and rooted at a leaf~$ x $. Assume that~$ \vr\coloneqq\alpha(\ve_x) $ is nontrivial and nondegenerate, let~$ \vs_0\in\vS $ emulate~$ \vr $ in~$ \vS $ for~$ \cF $, and consider~$ \alpha'\coloneqq\alpha_{x,\vs_0} $. Then~$ (T,\alpha') $ is an order-respecting~$ S $-tree in which~$ \menge{\sv_0} $ is a star associated with~$ x $ but with no other leaf of~$ T $. Moreover~$ \alpha'(t)\in\cF $ for all~$ t\ne x $ with~$ \alpha(t)\in\cF $.
\end{LEM}
The only difference in the statement between \cref{lem:shift_S-tree} and \cite{TangleTreeAbstract}*{Lemma 4.2} is that \cite{TangleTreeAbstract}*{Lemma 4.2} requires that $(T,\alpha)$ is an $S$-tree \emph{over $\cF$}, whereas we only require $(T,\alpha)$ to be an $S$-tree \emph{over some set of stars}. Consequently, in \cite{TangleTreeAbstract}*{Lemma 4.2} it is shown that then $ (T,\alpha') $ is an $S$-tree over $\cF\cup \{\{\sv_0\}\}$ whereas we only conclude that $\alpha'(t)\in \cF$ whenever $\alpha(t)\in \cF$.

\subsection{Splices in submodular universes}
In addition to the existing terminology, we shall need the following new concept, which has already been considered in \cites{TangleTreeGraphsMatroids}, but has not been given a name there:
In a submodular universe $\vU$ a separation $\vs$ is a \emph{splice for} a separation $\vr$ with $\vr \leq \vs$ if there is no separation $\vt$ with $\vr \leq \vt \leq \vs$ and $|t| < |s|$.
A \emph{splice between} two separations $\vr$ and $\vs$ with $\vr \leq \vs$ is one of minimum order among all $\vt$ with $\vr \leq \vt \leq \vs$.

These splices are good choices for proving separability due to the next lemma. It follows directly from the proof of Lemma 3.4 of \cite{TangleTreeGraphsMatroids} which, phrased in our terminology, considers a splice between two separations. We recapitulate the main argument of this proof below.

\begin{LEM}[\cite{TangleTreeGraphsMatroids}]\label{lem:splice_shift}
    Consider $\vS_k \subseteq \vU$ in a submodular universe.
  If $\vs \in \vS_k$ is a splice for $\vr \in \vS_k$ then, for every $\vt \in \vU$ with $\vt \geq \vr$,
  the order of $\vt \join \vs$ is at most the order of $\vt$.
  In particular, $\vs$ emulates $\vr$ in $\vS_k$.
\end{LEM}
\begin{proof}[Proof sketch, see \cite{TangleTreeGraphsMatroids}*{Lemma 3.4}]
 If the order of $\vt\join \vs$ were greater then the order of $\vt$ then, by submodularity, the order of $\vt\meet\vs$ would be less than the order of $\vs$. However, by the fish \cref{lem:fish}, $\vr\le\vt\meet\vs\le\vs$ and this contradicts the fact that $\vs$ is a splice for~$\vr$.
\end{proof}

This lemma then directly implies the ultimate statement of \cite{TangleTreeGraphsMatroids}*{Lemma 3.4}:

\begin{LEM}[\cite{TangleTreeGraphsMatroids}*{Lemma 3.4}]\label{lem:Sk_separable}
  Every $\vS_k \subseteq \vU$ in a submodular universe is separable.
\end{LEM}

\section{Structurally submodular separation systems}\label{sec:finite_merging_submodular}

In this section we will prove the first tree-of-tangles theorem of this paper.
It is a theorem for regular profiles, all of the same structurally submodular separation system, and states as follows:

\begin{restatable}{THM}{ToTviaTTD}\label{thm:ToTviaTTD}
	Let $ \vS $ be a structurally submodular separation system. Then $ S $ contains a nested set that distinguishes the set of regular profiles of~$ S $.
\end{restatable}

By itself~\cref{thm:ToTviaTTD} is nothing special; indeed, it is a slight weakening of~\cref{thm:Daniel}, which asserts the same but without requiring the profiles to be regular. In this case the ingredients of the proof are more interesting than its result: we shall obtain~\cref{thm:ToTviaTTD} as a direct consequence of~\cref{thm:duality}.

So let $ \vS $ be a structurally submodular separation system inside some universe~$ \vU $. Since we are interested in the regular profiles of~$ S $ we may assume that $ S $ has no degenerate elements. Our strategy will be as follows: we shall construct a set $ \cF\sub 2^{\vU} $ for which there is no $ \cF $-tangle of $ S $, and such that every element of $ \cF $ is included in at most one regular profile of~$ S $. If we can achieve this then~\cref{thm:duality} applied to this set~$ \cF $ will yield an $ S $-tree over~$ \cF $. The set $ N $ of edge labels of this $ S $-tree~$ (T,\alpha) $ will then be the desired nested set distinguishing all regular profiles of~$ S $: each regular profile $ P $ of~$ S $ orients the edges of~$ T $ and hence includes a star~$ \sigma $ of the form $ \alpha(t) $ for some $ t\in V(T) $. By choice of~$ \cF $ this~$ \sigma $ is included in no other regular profile of~$ S $, which means that it distinguishes~$ P $ from all other profiles.

To construct this set~$ \cF $, first let~$ \cP $ be the set of all `profile triples' in $ \vS $: the set of all~$ {\menge{\vr,\vs,(\vr\join\vs)^*}\sub\vS} $. For a consistent orientation of~$ S $ it is then equivalent to be a profile of~$ S $ and to be a~$ \cP $-tangle. Furthermore let $ \cC $ be the set of all $ \menge{\vs} $ with $ \vs\in\vS $ co-small. Finally, let $ \cM $ consist of each of the sets $ \max P $ of maximal elements of $ P $ for each regular profile $ P $ of~$ S $. We then take
\[ \cF\coloneqq\cP\cup\cC\cup\cM\,. \]
With these definitions the regular profiles of $ S $ are precisely its $ (\cP\cup\cC) $-tangles; and there are no $ \cF $-tangles of $ S $ since each regular profile $ P $ of~$ S $ includes~$ \max P\in\cM\sub\cF $. If this $\cF$ were a set of of stars and if we could feed this $ \cF $ to~\cref{thm:duality}, we would receive an~$ S $-tree over~$ \cF $ and the edge labels of this $ S $-tree would be our desired nested set, since each element of $ \cF $ in included in at most one regular profile of~$ S $: indeed, the regular profiles of $ S $ have no subsets in $ \cP $ or $ \cC $, and each element $ \max P\in\cM $ in included only in~$ P $ itself.

Unfortunately, we are still some way off from plugging $ \cF $ into~\cref{thm:duality}: we need to ensure that~$ \cF $ is a set of stars that is standard for~$ S $ and that~$ S $ is~$ \cF $-separable. Out of these the second and one half of the third are easy:~$ \cF $ is standard for~$ S $ since~$ \cC\sub\cF $ is, and~$ S $ is separable by~\cref{lem:separability}.

We thus need to show that~$ S $ is not only separable but~$ \cF $-separable. Unfortunately our current set~$ \cF $ is not even a set of stars yet. However, in~\cite{ProfileDuality} a solution was laid out for this exact situation: a series of lemmas from~\cite{ProfileDuality} shows that we can simply \emph{make} $ \cF $ a set of stars and close it under shifting without altering the set of $ \cF $-tangles of~$ S $.

The way to do this is as follows. Given two elements~$ \vr $ and~$ \vs $ of some set $ \sigma\sub\vS $, by submodularity, either $ \vr\meet\sv $ or $ \rv\meet\vs $ must lie in~$ \vS $. \emph{Uncrossing $ \vr $ and $ \vs $ in $ \sigma $} then means to replace either $ \vr $ by $ \vr\meet\sv $ or $ \vs $ by $ \rv\meet\vs $, depending on which of these two lies in~$ \vS $. (Structural submodularity ensures that at least one of them does.) Uncrossing all pairs of elements of $ \sigma $ in turn yields a star $ \sigma^* $, which we call an~\emph{uncrossing} of~$ \sigma $. (Note that $ \sigma^* $ is not in general unique since it depends on the order in which one uncrosses the elements of~$ \sigma $.) It is then easy to see that a regular profile of $ S $ includes $ \sigma $ if and only if it includes~$ \sigma^* $:
\begin{LEM}[\cite{ProfileDuality}*{Lemma~11}]\label{lem:uncrossing}
	If a regular profile of~$ S $ includes an uncrossing of some set, it also includes that set.
	
	Conversely, if a regular consistent orientation of~$ S $ includes some set, it also includes each uncrossing of that set.
\end{LEM}

Let us write $ \cF^* $ for the set of all uncrossings of elements of~$ \cF $. Then $ \cF^* $ is a set of stars that is standard for~$ S $. We are still not done, however, since $ \cF^* $ need not be closed under shifting. We can fix this in a similar manner though.

Just as for uncrossings it is not hard to show that the inclusion of a star's shift in a regular profile implies that star's inclusion:

\begin{LEM}[\cite{ProfileDuality}*{Lemma~13}]\label{lem:shift}
	If a regular profile of~$ S $ includes a shift of some star, it also includes that star.
\end{LEM}

In~\cite{ProfileDuality} the definition of a shift of a star contains additional technical assumptions on $ \sigma $ and $ \vs_0 $, keeping in line with the precise assumptions of~\cref{thm:duality}. However the proof of~\cref{lem:shift} does not necessitate this, and neither does its application.

\cref{lem:shift} says that if we close~$ \cF^* $ under shifting we, again, do not alter the set of~$ \cF^* $-tangles of~$ S $. Formally, set $ \cG_0=\cF^* $, and for $ i\ge 1 $ let $ \cG_i $ be the set of all shifts of star in~$ \cG_{i-1} $. Write~$ \hat{\cF^*}\coloneqq\bigcup_{i\in\N}\cG_i $. Then by~\cref{lem:shift} the $ \hat{\cF^*} $-tangles of~$ S $ are precisely its~$ \cF^* $-tangles, which is to say that there are no~$ \hat{\cF^*} $-tangles of~$ S $. Moreover this set $ \hat{\cF^*} $ still has the property that each star in it is included in at most one regular profile: let us say that~$ \hat{\sigma^*}\in\hat{\cF^*} $ \emph{originates from} $ \sigma\in\cF $ if $ \hat{\sigma^*} $ can be obtained by a series of shifts from an uncrossing of~$ \sigma $. Lemmas~\ref{lem:uncrossing} and~\ref{lem:shift} then say that if $ \hat{\sigma^*}\sub P $ for a regular profile $ P $, and $ \hat{\sigma^*} $ originates from~$ \sigma\in\cF $, then~$ \sigma\sub P $. Since the only element of~$ \cF $ which~$ P $ includes is~$ \max P $, this implies that no other regular profile of~$ S $ includes~$ \hat{\sigma^*} $.

We can thus formally prove~\cref{thm:ToTviaTTD}:

\begin{proof}[Proof of~\cref{thm:ToTviaTTD}]
	Define $ \cP $, $ \cC $, $ \cM $, $ \cF $, $ \cF^* $, and $ \hat{\cF^*} $ as above. Then $ \hat{\cF^*} $ is standard for~$ S $ since $ \cC\sub\hat{\cF^*} $, and closed under shifting by construction. By~\cref{lem:separability}~$ S $ is separable. Together this gives that $S$ is $\cF$-separable. Hence we can apply the tangle-tree duality theorem~\ref{thm:duality} to obtain either an~$ \hat{\cF^*} $-tangle of~$ S $ or an~$ S $-tree over~$ \hat{\cF^*} $.
	
    We claim that the first is impossible. For suppose that~$ P $ is some~$ \hat{\cF^*} $-tangle of~$ S $. From $ \cC\sub\hat{\cF^*} $ we know that $ P $ is a regular and consistent orientation of~$ S $. If~$ P $ has the profile property~\eqref{property:P} then we could derive a contradiction from Lemmas~\ref{lem:uncrossing} and~\ref{lem:shift} since $ S $ has no~$ \cF $-tangle. On the other hand, if $ P $ is not a profile, then $ P $ includes some set~$ \sigma\in\cP $. By the second part of~\cref{lem:uncrossing} $ P $ then also includes some (in fact: each) uncrossing of~$ \sigma $ and hence a set in~$ \cF^*\sub\hat{\cF^*} $, contrary to its status as an~$ \hat{\cF^*} $-tangle.
	
	So let $ (T,\alpha) $ be the $ S $-tree over~$ \hat{\cF^*} $ returned by~\cref{thm:duality}, which we may assume to be irredundant~(\cref{lem:TreeSets_6.2}). Let $ \vN $ be the image of~$ \alpha $. Then~$ N $ is a nested subset of~$ S $~(\cref{lem:TreeSets_6.3}).
	Let us show that $ N $ distinguishes all regular profiles of~$ S $. Since~$ (T,\alpha) $ is an~$ S $-tree over~$ \hat{\cF^*} $ each consistent orientation of~$ S $ includes some star $ \hat{\sigma^*}\in\hat{\cF^*}\cap 2^\vN $. In particular if $ P $ is a regular profile of~$ S $ then $ P $ includes some~$ \hat{\sigma^*}\in\hat{\cF^*}\cap 2^\vN $. Since the only element of $ \cF $ which $ P $ includes is~$ \max P $, this~$ \hat{\sigma^*} $ must originate from~$ \max P $. Consequently no other regular profile of~$ S $ can include~$ \hat{\sigma^*} $, since none of them include~$ \max P $. Thus $ \hat{\sigma^*} $ distinguishes $ P $ from every other regular profile of~$ S $. Since~$ P $ was arbitrary this shows that $ N $ distinguishes all regular profiles of~$ S $.
\end{proof}

Let us make some remarks on this proof of~\cref{thm:ToTviaTTD}. First, in the definition of~$ \cF $, we could have used other sets~$ \cM $: the only properties of~$ \cM $ that every regular profile of~$ S $ contains some set from~$ \cM $, and that no element of~$ \cM $ is included in more than one such regular profile. We will put this observation to good use in \cref{sec:finite_merging_degrees}, where we will make a more refined choice for~$ \cM $ than simply collecting the sets of maximal elements from each profile.

Second, with the approach shown here it is not easy to strengthen~\cref{thm:ToTviaTTD} to the level of~\cref{thm:Daniel} by dropping the assumption of regularity, since~\cref{lem:shift} cannot do without this regularity.

In the remainder of this section we will show a more direct version of the proof presented above.
This proof will be the guiding principle by which we will approach the issues of efficiency and profiles of differing order in \cref{sec:efficiency,sec:different_order}.

The core idea is that one can take as $ \cF $ the set of all stars that are included in at most one regular profile of~$ S $.
An $ S $-tree over this set~$ \cF $ would immediately lead to the desired nested set distinguishing all regular profiles. Moreover this~$ \cF $ is standard for~$ S $ since~$ \cC\sub\cF $. To obtain this $ S $-tree over~$ \cF $ from~\cref{thm:duality} one would only need to show two things, namely that $\vS$ is $\cF$-separable and that there is no~$ \cF $-tangle of~$ S $. The first of these amounts to~\cref{lem:shift}; the second requires the two insights that every $ \cF $-avoiding consistent orientation is a regular profile, and that each regular profile of~$ S $ includes some star in~$ \cF $, both of which retrace some steps of~\cref{lem:uncrossing}.

\begin{LEM}\label{lem:profile-star}
    Let $\vS \sub \vU$ be a structurally submodular separation system and let $P$ be a profile of $\vS$.
    There exists a star $\sigma \subseteq P$ such that no other profile of $\vS$ includes~$\sigma$.
\end{LEM}
\begin{proof}
    Let $\sigma \sub P$ be a star which minimizes the number of profiles which include~$\sigma$.
    Suppose for a contradiction that there exists a profile $P' \neq P$ with $\sigma \sub P$.
    Some separation $s$, say, distinguishes $P$ from $P'$. Clearly $s$ crosses some element of~$\sigma$.

    Suppose that, subject to the above, $\sigma$ and $s$ are chosen such that the number of separations in $\sigma$ that $s$ crosses is minimum.
    Let $\vt \in \sigma$ be a separation that $s$ crosses.
    If either of the corner separations $\vt\join\vs$ or $\vt\join\sv$ was in $\vS$ then, by the profile property, it would distinguish $P$ and $P'$.
    It would also, by the fish~\cref{lem:fish}, cross one less separation in $\sigma$ than $s$ does, contradicting the choice of~$s$.

    So by submodularity the corner separations $\vt\meet\vs$ and $\vt\meet\sv$ are in $\vS$.
    Note that, by the profile property, any profile including \[
        \sigma' \coloneqq \sigma \sm \menge{\vt} \cup \menge{\vt\meet\vs,\, \vt\meet\sv}
    \] also includes~$\sigma$.
    Consequently $\sigma'$ together with $s$ are a better choice than $\sigma$ and $s$, a~contradiction.
\end{proof}

\begin{LEM}\label{lem:non_profiles}
 Given any set $\cP$ of profiles of $\vS$, every consistent orientation $O$ of $\vS$ which is not a profile in $\cP$ contains a star $\sigma$ which is not contained in any profile in $\cP$.
\end{LEM}
\begin{proof}
 Since $O$ is not a profile in $\cP$ there is, for every profile $P$ in $\cP$, a separation $s$ such that $\vs\in O$ but $\sv\in P$. Pick a set $\vN\subseteq O$ which contains one such separation for every profile in $\cP$ and is, subject to this, $\le$-minimal: That is, there is no other such set $N'$ together with an injective function $\alpha:N'\to N$ satisfying $\vs'\le \alpha(\vs')$ for all $\vs'\in N'$.
 
 If $N$ is a nested set, then $N$ contains the desired star, so suppose that $\vs,\vt\in N$ cross. By submodularity we may suppose, after possibly renaming $\vs$ and $\vt$, that $\vs\meet\tv\in S$ and thus, by consistency, $\vs\meet\tv\in O$. We claim that $N\sm \{\vs\}\cup \{\vs\meet\tv\}$ is also a candidate for $N$, contradicting the $\le$-minimality. So suppose that $N\sm \{\vs\}\cup \{\vs\meet\tv\}$ does not contain a separation $\vr$ such that $\rv\in P$, say. Then clearly $\sv\in P$ and $\vt\in P$, thus, by the profile property $\sv\join \vt\in P$ which is precisely such an $\rv$, a contradiction.
\end{proof}

We are now ready to give a proof of \cref{thm:ToTviaTTD} without resorting to \cref{lem:uncrossing}:
\ToTviaTTD*
\begin{proof}[Direct Proof] Let $\cP$ be the set of regular profiles of $S$.
    Let $\cF_\cP \sub 2^\vS$ consist of all stars $\sigma \sub \vS$ for which one of the following is true:
    \begin{enumerate}[label={(\roman*)}]
        \item No profile in $\cP$ includes $\sigma$, or
        \item Precisely one profile in $\cP$ includes $\sigma$.
    \end{enumerate}
    This $\cF_\cP$ is, by \cref{lem:shift}, closed under shifting: any shift of a star contained in at most one profile  is again contained in at most one profile.
    The set $\cF_\cP$ is also standard for $\vS$, since cosmall separations are contained in no regular profile.

    By \cref{thm:duality} there either exists an $S$-tree over $\cF_\cP$, or an $\cF_\cP$-tangle of $S$.
    In the former case we obtain the desired nested set.
    For the latter case observe that every $\cF_\cP$-tangle $P$, say, is a regular profile: By \cref{lem:non_profiles} every consistent orientation which avoids $\cF_\cP$ is a profile and if $P$ would not be regular, it would contain a cosmall separation $\vs$ which is impossible, since $\{\vs\}\in \cF_\cP$.
    So by \cref{lem:profile-star} there exists a star $\sigma\sub P$ which every profile other than $P$ avoids.
    In particular $\sigma \in \cF_\cP$, which contradicts the fact that $P$ is an $\cF_\cP$-tangle.
\end{proof}

\section{Application: Degrees in trees of tangles}\label{sec:finite_merging_degrees}
In this section we are going to see that our proof of \cref{thm:ToTviaTTD} in \cref{sec:finite_merging_submodular} has one advantage over the usual, more direct proofs of \cref{thm:ToTviaTTD} from \cite{AbstractTangles,FiniteSplinters}: It allows us to easily control the maximum degree of the resulting tree. More precisely:
Let $ S $ be a structurally submodular separation system and $ P $ a regular profile of $ S $. In this section we answer the following question: over all trees of tangles that distinguish all regular profiles of $ S $, how low can the degree of the node containing $ P $ in those trees of tangles be?

Let us first make this notion of degree in a tree of tangles formal. For the purposes of this application only, a \emph{tree of tangles (for~$ S $)} is an irredundant $ S $-tree $ (T,\alpha) $ whose set of edge labels distinguishes all regular profiles of~$ S $. For a regular profile~$ P $ of~$ S $ and a tree of tangles~$ (T,\alpha) $, the \emph{node of~$ P $ in~$ T $} is the unique sink of the orientation of~$ T $'s edges induced by~$ P $, and the \emph{degree of~$ P $ in~$ (T,\alpha) $} is the degree of this node.

Our question is thus: what is the minimum degree of $ P $ in $ (T,\alpha) $ over all trees of tangles~$ (T,\alpha) $?

A lower bound for this degree can be established as follows. Let $ \delta(P) $ denote the minimal size of a set of separations which distinguishes~$ P $ from all other regular profiles of~$ S $. If $ t $ is the node of~$ P $ in some tree of tangles~$ (T,\alpha) $ then~$ \alpha(t) $ is such a set of separations which distinguishes~$ P $ from all other regular profiles of~$ S $; thus, the degree of~$ P $ in every tree of tangles~$ (T,\alpha) $ is at least~$ \delta(P) $.

We show that this lower bound can be achieved: there is a tree of tangles $ (T,\alpha) $ for~$ S $ in which~$ P $ has degree exactly~$ \delta(P) $. In fact~$ (T,\alpha) $ will be optimal in this sense not just for~$ P $, but for all regular profiles of~$ S $ simultaneously. Additionally the degrees of those nodes of~$ (T,\alpha) $ that are not the node of some regular profile will not be unreasonably high: the maximum degree of~$ T $ will be attained in some profiles' node.

\begin{THM}\label{thm:Degree}
	Let $ S $ be a structurally separation system. Then there is a tree of tangles $ (T,\alpha) $ for~$ S $ in which each regular profile~$ P $ of~$ S $ has degree exactly~$ \delta(P) $. Furthermore, if~$ \Delta(T)>3 $, then~$ \Delta(T)=\delta(P) $ for some regular profile~$ P $ of~$ S $.
\end{THM}

To prove~\cref{thm:Degree} we will follow the first proof of~\cref{thm:ToTviaTTD}, making a more refined choice of~$ \cM $, and utilise the fact that uncrossing and shifting a set cannot increase its size.

We will later see an example of a structurally submodular separation system in which $ \delta(P)\le 2 $ for every profile $ P $ but $ \Delta(T)=3 $ for every tree of tangles $ T $; this will demonstrate that the last assertion of~\cref{thm:Degree} is optimal in that regard.

Observe further that the set of maximal elements of a profile~$ P $ is a set which distinguishes~$ P $ from every other profile of~$ S $. (In fact, the maximal elements of~$ P $ distinguish~$ P $ from every other consistent orientation of~$ S $.) Therefore~$ \delta(P)\le\abs{\max P} $ and hence the degree of~$ P $ in the tree of tangles from~\cref{thm:Degree} is at most~$ \abs{\max P} $.

Let us now prove~\cref{thm:Degree}:

\begin{proof}[Proof of~\cref{thm:Degree}]
	For each regular profile~$ P $ of~$ S $ pick a subset~$ D_P\sub P $ of size~$ \delta(P) $ which distinguishes~$ P $ from every other regular profile of~$ S $. Let~$ \cD $ be the set of these~$ D_P $. Define~$ \cP $ and~$ \cC $ as in the proof of~\cref{thm:ToTviaTTD}, and set
	\[ \cF\coloneqq\cP\cup\cC\cup\cD\,. \]
	From here, define~$ \cF^* $ and~$ \hat{\cF^*} $ just as in~\cref{thm:ToTviaTTD} and follow the same proof. The result is an $ S $-tree over~$ \hat{\cF^*} $, which we may assume to be irredundant and hence a tree of tangles for~$ S $.
	
	Now let~$ P $ be a regular profile of~$ S $, let $ t $ be the node of~$ P $ in~$ T $, and~$ \hat{\sigma^*}\coloneqq\alpha(t) $. As in the proof of~\cref{thm:ToTviaTTD} the only element of~$ \cF $ from which~$ \hat{\sigma^*} $ can originate is~$ D_P $. Since uncrossing and shifting~$ D_P $ cannot increase its size we have~$ \abs{\hat{\sigma^*}}\le\abs{D_P}=\delta(P) $. Conversely we have~$ \abs{\hat{\sigma^*}}\ge\delta(P) $ since~$ \hat{\sigma^*} $ distinguishes~$ P $ from all other regular profiles. Thus the degree of~$ P $ in~$ (T,\alpha) $ is indeed~$ \delta(P) $.
	
	Finally, if~$ \Delta(T)>3 $, the maximum degree of~$ T $ is attained in some node~$ t $ whose associated star~$ \alpha(t) $ originates from some~$ D_P\in\cD $, since all elements of~$ \hat{\cF^*} $ originating from elements of~$ \cP $ or~$ \cC $ have size at most three. As above we thus have $ \abs{\alpha(t)}\le\abs{D_P}=\delta(P)  $, giving~$ \Delta(T)=\delta(P) $.
\end{proof}

Let us see an example showing that we cannot guarantee to find $ T $ with maximum degree less than three even if all regular profiles of $ S $ have $ \delta(P)\le 2 $:

\begin{EX}\label{ex:degree}
	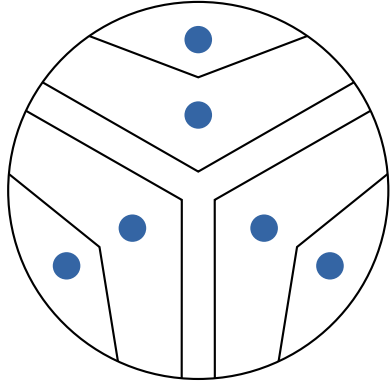
\begin{figure}[ht]
        \centering
        \begin{tikzpicture}[point/.style={circle,fill=skyblue},sep/.style={thick}]
            \draw[thick] (0,0) circle [radius=2.5cm];
            \foreach \angle in {-30, -150, -270} {
                \draw (\angle:1cm) node[point] {};
                \draw (\angle:2cm) node[point] {};
                \draw[sep] (\angle - 55 : 2.5cm) -- (\angle: .25cm) -- (\angle + 55: 2.5cm);
                \draw[sep] (\angle - 35 : 2.5cm) -- (\angle: 1.5cm) -- (\angle + 35: 2.5cm);
            }
        \end{tikzpicture}
		\caption{A ground-set and system of bipartitions.}
		\label{fig:degree}
	\end{figure}
	Let~$ V $ consist of the six points in~\cref{fig:degree}, and~$ S $ be the separation system given by the six outlined bipartitions of~$ V $ together with~$ \menge{\es,V} $. (That is, $\vS$ contains $(A,B)$ and $(B,A)$ for each of these bipartions $\{A,B\}$. We have $(A,B)\leq (C,D) :\Leftrightarrow A \subseteq C$, and $(A,B)^* = (B,A)$. Compare \cite{TangleTreeGraphsMatroids}.) The regular profiles of~$ S $ correspond precisely to the six elements of~$ V $: each~$ v\in V $ induces a profile of~$ S $ by orienting each bipartition towards~$ v $, and conversely each profile of~$ S $ is of this form. Each profile~$ P $ has at most two maximal elements, giving~$ \delta(P)\le 2 $. However, every tree of tangles for~$ S $ must contain the outer three bipartitions and hence have a maximum degree of at least three.
\end{EX}

\section{Efficient distinguishers}\label{sec:efficiency}
Often our structurally submodular separation system $\vS$ is actually an $\vS_k$, the set of all separations of order less than $k$, of some submodular universe $\vU$.
In this case we are not just interested in a nested set of separations which distinguishes all profiles, but one which does so \emph{efficiently}, that is, for any two profiles it contains a distinguishing separation of minimum possible order.
In this section we are going to see how this can be achieved for regular profiles of a fixed $\vS_k$
utilising the duality theorem together with a separate application of its core mechanism: shifting $S$-trees.

We will prove this theorem:
\begin{THM}\label{thm:efficient}
    Let $\vU$ be a submodular universe and let $\cP$ be a set of regular profiles of~$\vS_k$.
    Then there exists a nested set $N \sub S_k$ efficiently distinguishing all the profiles in~$\cP$.
\end{THM}
Our approach is similar to the one of the direct proof in \cref{sec:finite_merging_submodular}, but we shall restrict our set of stars so that they do not interfere with efficiency.

Consider a nested set of separations which distinguishes all profiles efficiently and, subject to this, is $\subseteq$-minimal.
Every profile $P$ induces an orientation of this set, and the maximal elements of this orientation form a star.
The separations in this star are, in a way, `well connected' to the profile.
We make this a condition on the stars we consider.
\newcommand{\linkedto}{\operatorname{Eff}}
For a star $\sigma$ and a profile $P$, we say that \emph{$\sigma$ has the property $\linkedto(P)$} if the following holds:
\begin{equation*}
    \nexists\, \vs\in\sigma\text{ and }\vs'\in P \colon \vs \le \vs' \text{ and } \abs{s'} < \abs{s}. \tag{$\linkedto(P)$}\label[property]{prop:prof_eff}
\end{equation*}
\newcommand{\propertyEff}[1]{\hyperref[prop:prof_eff]{property $\linkedto(#1)$}}
\newcommand{\Eff}[1]{\hyperref[prop:prof_eff]{$\linkedto(#1)$}}
This condition ensures that, for two profiles $P$ and $P'$, a star $\sigma$ with \cref{prop:prof_eff} containing~$\vs$, and a star $\sigma'$ with \propertyEff{P'} containing $\sv$, the separation $\vs$ needs to be an efficient $P$--$P'$-distinguisher. For if $ \vs $ is not efficient, consider an efficient $P$--$P'$-distinguisher $\vr\in P$. Then $\vr$ cannot be nested with $\vs$, since $\vs\le\vr$ would contradict \cref{prop:prof_eff} whereas $\vr\le\vs$ would contradict \propertyEff{P'}. But $\vr$ cannot cross $\vs$ either: if it did we would have either $\abs{\vr\join \vs}<\abs{\vs}$ or $\abs{\vr\meet\vs}<\abs{\vs}$ by submodularity, again contradicting \cref{prop:prof_eff} or \Eff{P'}, respectively.

\Cref{prop:prof_eff} is preserved under taking shifts:

\begin{LEM}\label{lem:efficient_shift}
    Let $\vs\in\vS_k$ be a splice for $\vr\in\vS_k$ and let $\sigma \subseteq \vS_k$ be a star with some $\vx \in \sigma$ with $\vx \geq \vr$.
    If a profile $P$ contains both $\sigma$ and $\sigma' \coloneqq \sigma_{\vx}^{\vs}$ and $\sigma$ has \cref{prop:prof_eff} then also $\sigma'$ has \cref{prop:prof_eff}.
\end{LEM}
\begin{proof}
    Suppose for a contradiction that $\sigma'$ does not have \cref{prop:prof_eff}, that is,
    above some $\vt \meet \sv \in \sigma'$, where $\vt \in \sigma$, there is a separation $\vt'\in P$ of lower order than $\vt\meet\sv$.

    We will first show that we may assume $\vt' \leq \vt$.
    Since $\vs$ is a splice for $\vr$ we have $|\vs \meet \tv| \ge |\vs|$, and thus by submodularity
    $|\sv \meet \vt| \leq |\vt|$.
    So if $\vt' > \vt$ then this contradicts the assertion that $\sigma$ has \cref{prop:prof_eff}.
    If however $\vt'$ crosses $\vt$ then by the profile property of $P$ and \cref{prop:prof_eff} of $\sigma$ the supremum $\vt'\join \vt$ has at least the order of $\vt$. By submodularity then $\vt' \meet \vt$ has at most the order of $\vt'$.
    This is also a separation in $P$ which is above $\vt\meet\sv$ and of lower order than $\vt\meet\sv$, so we may consider it instead.

    Now, since $\vs$ is a splice for $\vr$ we have that $|\tv' \meet \vs| \geq |\vs|$, so by submodularity $\vt' \meet \sv$ has at most the order of $\vt'$.
    But this $\vt' \meet \sv$ is the same as $\vt \meet \vs$ since $\vt \geq \vt' \geq \vt \meet \sv$.
    So we have $|\vt \meet \vs| \leq |\vt'|$, which contradicts the assumption that $|\vt'| < |\vt \meet \vs|$.
\end{proof}
We define $\cF_e$ as the set of all stars $\sigma \subseteq \vS_k$ which are contained in at most one profile in $\cP$ and which, if they are contained in a profile $P\in \cP$, fulfil \cref{prop:prof_eff}. 

From \cref{lem:shift,lem:efficient_shift} immediately we obtain the following corollary:

\begin{COR}\label{cor:efficient_closed}
    $S_k$ is $\cF_e$-separable.
\end{COR}

However, an $S$-tree over $\cF_e$ does not necessarily give rise to an efficient distinguisher set for $\cP$ because we make no assumptions on those stars which are not contained in any profile.
Our proof of \cref{thm:efficient} will need to make additional arguments on why an efficient such tree exists.

It would be much more elegant if we could introduce a condition, similar to \Eff{\cdot}, on the stars which are in no profile, so as to guarantee that any $S_k$-tree over these stars is as desired.
However, all possible such properties that the authors could come up with failed to give $\cF$-separability and
there is reason to believe that such a solution is not possible:
The critical part in the proof of \cref{thm:efficient} will make a global argument, specifically that of two shifts of one separation one is an efficient distinguisher.
Separability on the other hand is defined in terms of each individual shift of a star.

For this section's analogue of \cref{lem:profile-star},
we define the \emph{fatness} of a star $\sigma$ as the tuple $(n_{k-1},n_{k-2},\dots,n_1,n_0)$, where $n_i$ is the number of separations of order $i$ in $\sigma$. We will consider the lexicographic order on the fatness of stars.

\begin{LEM}\label{lem:profile-star-efficient}
   Given a set $\cP$ of regular profiles of $\vS_k$, every profile $P\in \cP$ includes a star in $\cF_e$.
\end{LEM}
\begin{proof}
    By \cref{lem:profile-star} $P$ includes a star which is contained only in $P$.
    Take such a star $\sigma$ which has lexicographically minimal fatness and
    suppose for a contradiction that $\sigma$ does not have \cref{prop:prof_eff}.
    So take $\vs \in \sigma$ and $\vr \in P$ with $\vs\le\vr$ and $\abs{r} < \abs{s}$.
    Among the possible choices for $\vr$, let $\vr$ be one which crosses as few separations in $\sigma$ as possible.
    If $\vr$ were nested with $\sigma$ then the maximal elements of $\sigma \cup \menge{\vr}$ would form a star of lower fatness,
    thus we may suppose that $\vr$ crosses some $\vx\in\sigma$.

    By the choice of $\vr$, the corner separations $\vr\join\vx$ and $\vr\meet\xv$ must have strictly higher order than $\abs{r}$ since both are $\geq \vs$.
    Thus, by submodularity, the corner separations $\vr\meet\vx$ and $\rv\meet\vx$ have strictly lower order than $\abs{x}$.
    Now the star $\sigma' \coloneqq \sigma \sm \menge{\vx} \cup \menge{\vr\meet\vx,\, \rv\meet\vx}$ has a lower fatness.
    This star is still contained in $P$ by consistency and in no other profile, since every profile which includes $\sigma'$ also includes $\sigma$ by the profile property applied with $\vx$ and~$\vr$.
    This contradicts the choice of $\sigma$.
\end{proof}

We are now able to prove \cref{thm:efficient}:
\begin{proof}[Proof of \cref{thm:efficient}]
We may apply \cref{thm:duality} for $\cF_e$ since $\vS_k$ is $\cF_e$-separable by \cref{cor:efficient_closed} and $\cF_e$ is standard since cotrivial separations are not contained in any regular profile.
From this theorem we cannot get an $\cF_e$-tangle: such a tangle cannot be a profile in $\cP$ by \cref{lem:profile-star-efficient}, and \cref{lem:non_profiles} states that every consistent orientation which is not a profile in $\cP$ includes a star which is not contained in any profile in $\cP$, but each of these stars is contained in $\cF_e$, so no such orientation is an $\cF_e$-tangle.
So instead, there exists an $S_k$-tree over $\cF_e$.

Among all $S_k$-trees over $\cF_e$ pick an irredundant one, $(T,\alpha)$ say, whose associated separations efficiently distinguishes as many pairs of profiles as possible.
Let us suppose that some pair of profiles $P_1,P_2$ is not distinguished efficiently by this tree.

Consider the nodes $v_{P_1}, v_{P_2}$ of this tree corresponding to $P_1$ and $P_2$. These nodes are distinct, since every star in $\cF_e$ is contained in at most one profile.
Moreover, we can assume without loss of generality, that in no node on the path between $v_{P_1}$ and $v_{P_2}$ there lives a profile $Q$: In that case either the pair $P_1,Q$ or the pair $Q,P_2$ would not be efficiently distinguished by $(T,\alpha)$ either, so we could consider them instead.

Let $\vs_{P_1}$ be the separation associated to the first edge on the path from $v_{P_1}$ to $v_{P_2}$ and let $\vs_{P_2}$ be the separation associated to the first edge on the path from $v_{P_2}$ to $v_{P_1}$.
There exists a separation $t$ which efficiently distinguishes $P_1$ and $P_2$ and is nested with $\vs_{P_1}$ and $\vs_{P_2}$: if $\vt\in P_1$ is not nested with, say $\vs_{P_1}$, we know by \cref{prop:prof_eff} that $\sv_{P_1}\join \vt$ needs to have order at least $|s_{P_1}|$, thus $\sv_{P_1}\meet \vt$ has order at most $|t|$, so it efficiently distinguishes $P_1$ and $P_2$ and is nested with $\vs_P$. Thus by the fish \cref{lem:fish}, there indeed needs to exists such a $t$ which efficiently distinguishes $P_1$ and $P_2$ and is nested with $\vs_{P_1}$ and $\vs_{P_2}$.
Moreover, $t$ has an orientation such that  $\vs_{P_1}\le \vt \le \sv_{P_2}$, otherwise the existence of $t$ again contradicts either \cref{prop:prof_eff} or \Eff{Q}.
Note that $\vt$ thus is a splice between $\vs_{P_1}$ and $\sv_{P_2}$
and therefore $\vt$ emulates $\vs_{P_1}$ for $\cF_e$ and $\tv$ emulates $\vs_{P_2}$ for $\cF_e$. 

Let $T_{P_1}$ be the subtree of $T$ consisting of the component of $T-v_{P_1}$ which contains $v_{P_2}$ together with $v_{P_1}$ and similarly let $T_{P_2}$ be the subtree consisting of the component of $T-v_{P_2}$ containing $v_{P_1}$ together with $v_{P_2}$.

We consider the trees $(T_{P_1}, \alpha_{P_1})$ and $(T_{P_2}, \alpha_{P_2})$ obtained from $(T_{P_1}, \alpha\restricts T_{P_1})$ and $(T_{P_2}, \alpha\restricts T_{P_2})$ by applying the shifts $\shifting{\vs_{P_1}}{\vt}$ and $\shifting{\vs_{P_2}}{\tv}$, respectively. Consider now the tree $(T',\alpha')$ obtained from these two trees by identifying the respective edges associated with~$\vt$. By applying \cref{lem:shift_S-tree} with the two shifted trees the combined tree is again over~$\cF_e$. We may again assume it to be irredundant. We are going to show that it efficiently distinguishes more pairs of profiles than $(T, \alpha)$.

Let $Q_1,Q_2$ be a pair profiles which were efficiently distinguished by a separation $\vr$ associated to an edge of $(T,\alpha)$.
If $\vr$ is not associated to any edge of $(T',\alpha')$ then, without loss of generality, either $\vs_{P_1}\le \vr \le \vs_{P_2}$ or both $\vs_{P_1} \le \vr$ and $\vs_{P_2}\le \vr$.

In the first case $\vr$ distinguishes $P_1$ and $P_2$ and therefore $|\vr| > |\vt|$. By the definition of the shift, our tree $(T',\alpha')$ contains both, $\vr\join \vt$ and $\vr\meet\vt$, and both of them have order at most the order of $\vr$, by \cref{lem:splice_shift}. However, one of $\vr\join \vt, \vr\meet\vt$ and $\vt$ distinguishes $Q_1$ and $Q_2$ and does so efficiently.

In the second case, by the definition of the shift, our tree $(T',\alpha')$ contains both, $\vr\join \vt$ and $\vr\join\tv$, and both of them have order at most the order of $\vr$, again by \cref{lem:splice_shift}. Again, one of $\vr\join \vt$ and $\vr\join\tv$ distinguishes $Q_1$ and $Q_2$ and does so efficiently.

Thus, since $(T',\alpha')$ additionally efficiently distinguishes $P_1$ and $P_2$ with $t$, this contradicts the choice of $(T,\alpha)$.
\end{proof}

\section{Degrees in efficient trees of tangles}\label{sec:efficient_degrees}
In this section we apply our method from \cref{sec:finite_merging_degrees} to \cref{thm:efficient} to obtain a tree of tangles of low degree, but this time one which efficiently distinguishes the profiles.
That is, we are interested in the minimal degrees of a tree of tangles whose associated separations efficiently distinguish all regular profiles of~$S_k$.

Extending the definitions of \cref{sec:finite_merging_degrees}, let us say that a tree of tangles $(T,\alpha)$ for~$S_k$ is \emph{efficient}, if the set of edge labels not only distinguishes all regular profiles of~$S_k$, but does so efficiently.

Given a $k$-profile $P$, we denote by $\delta_e(P)$ the minimal size of a star $\sigma\subseteq P$ with \cref{prop:prof_eff} which distinguishes $P$ from all other regular profiles of $S_k$, i.e., every other regular profile orients some $\vs\in \sigma$ as $\sv$.
Note that, by \cref{lem:profile-star-efficient}, there exists such a star for every regular profile $P$, thus $\delta_e(P)$ is a well defined natural number.

We denote by $\delta_{e,\max}$ the maximum of $\delta_e(P)$ over all regular profiles $P$.

We can give a bound on $\delta_e(P)$ which is not in terms of stars or nested sets:
\begin{LEM} Let $P$ be a regular $k$-profile in $U$ and let $D_P\subseteq P$ be a subset of $P$ which contains, for every regular $k$-profile $P'\neq P$ in $U$, a separation which efficiently distinguishes $P$ from $P'$. Let us denote as $m$ the number of maximal elements of $D_P$. Then $\delta_e(P)\le m$.
\end{LEM}
\begin{proof}
It is enough to consider a set $D_P\subseteq P$ such that $m = \abs{\max D_P}$ is as small as possible.
Moreover, we may assume without loss of generality that every element of $D_P$ distinguishes $P$ efficiently from some other profile in $\cP$, since we could otherwise remove it from $D_P$.
We may furthermore assume that, subject to all this, $D_P$ is chosen so that $\max D_P$ is $\le$-minimal.
Furthermore we may suppose that, for separations $\vr \leq \vs$ in $D_P$, the order of $\vr$ is lower than the order of $\vs$, since otherwise we could just remove $\vr$ from~$D_P$.

If the maximal separations in $D_P$ are pairwise nested, they satisfy \cref{prop:prof_eff} by the fact that they distinguish $P$ efficiently from some other profile $P'$.
Further, every profile $P'$ is distinguished from $P$ by some maximal separation in $D_P$: there is an efficient $P$-$P'$ distinguisher $\vs\in D_P$ and thus a maximal separation $\vt \ge \vs$ in $D_P$ also distinguishes $P$ from $P'$.
Hence, if the maximal elements of $D_P$ are pairwise nested, they are a candidate for $\delta_e(P)$ and therefore witness that $\delta_e(P)\le m$.

So suppose that this is not the case, so two maximal separations $\vs,\vt\in D_P$ cross and, without loss of generality, $|\vs|\le |\vt|$.
By the definition of $D_P$, there is a profile $P_s$ which is efficiently distinguished from $P$ by $\vs\in D_P$. Similarly, there is such a profile $P_t$ for $\vt$.

Since $D_P$ was chosen to have as few maximal elements as possible, the separation $\vs\join \vt$ has greater order than $t$: otherwise we could, by consistency and the profile property, replace $\vt$ in $D_P$ by $\vs\join \vt$.
Thus, by submodularity, the order of $\vs\meet \vt$ is less than the order of $\vs$. In particular, by efficiency of $s$ and $t$, neither $P_s$ nor $P_t$ contains $(\vs \meet\vt)^* = \sv\join \tv$.

Thus $\vs\meet \tv$ and $\sv\meet\vt$ have order precisely $|\vs|$ and $|\vt|$, respectively: if one of them had lower order this would, by the profile property, contradict the fact that $s$ or $t$, respectively, efficiently distinguishes $P$ from $P_s$ or $P_t$, respectively.
This means that, in particular, $\vs\meet\tv$ efficiently distinguishes $P$ from $P_s$.

For every $\vr \le \vs$ in $D_P$ we have assumed $|\vr| < |\vs|$. Both $\vr\meet\vt$ and $\vr\meet\tv$ have at most the order of $\vr$ due to submodularity, the efficiency of $\vt$, the profile property and consistency, analogue to the above.

Let us consider the set $D_P'$ obtained from $D_P$ by removing all $\vr\le\vs$, and adding $\vs\meet\tv$ as well as, for every $\vr\le\vs$, any $\vr\meet\vt$ and $\vr\meet \tv$ which efficiently distinguishes $P$ from some other profile.
By the above, this set $D_P'$ distinguishes $P$ from every other regular profile, and is a candidate for $D_P$.
The maximal separations of $D_P'$ and of $D_P$ are the same except that $\vs$ in $D_P$ is replaced by $\vs\meet\tv$ in $D_P'$.
This contradicts the choice of $D_P$ with $\le$-minimal maximal elements.
\end{proof}

To limit the degree of the node of $P$ in our tree of tangles we want to remove from $\cF_e$ all the stars which are contained in $P$ but are larger than $\delta_e(P)$.
In order to achieve a maximum degree of $\delta_{e,\max}$ we also need to limit the size of the stars in $\cF_e$ which are contained in no profile to $\delta_{e,\max}$.
As in \cref{sec:finite_merging_degrees}, we cannot limit the maximum degrees below 3.
Along the lines of the proof of \cref{lem:uncrossing}, the next lemma shows that we can find, in every consistent orientation $O$ of $\vS_k$ which is not a profile, a star of size $3$ contained in $O$ and in no profile.

\begin{LEM}\label{lem:profiles_stars}
Every consistent orientation $O$ of $\vS_k$ which is not a profile contains a star $\sigma$ of size $3$ which is not contained in any profile.
\end{LEM}
\begin{proof}
As $O$ is not a profile, there are $\vs,\vt\in O$ such that $\sv\meet\tv\in O$. By submodularity, either $\vs\meet\tv$ or $\sv\meet \vt\in \vS$, let us suppose the former one. Then $\sigma=\{\vs\meet\tv, \vt,\sv\meet\tv\}$ is a star in $O$ and $\sigma$ cannot be contained in any profile: any profile $P$ needs to contain either $\vs$ or $\sv$, and the profile property implies that $P$ then cannot contain both, $\vs\meet\tv$ and $\sv\meet\tv$.
\end{proof}

We can now show the following variant of \cref{thm:efficient}, which shows that we can find a tree of tangles of bounded degree:
\begin{THM}\label{thm:efficient_degrees}
    Let $\vU$ be a submodular universe and let $\cP$ be the set of regular profiles of~$\vS_k$.
    Then there exists tree of tangles $ (T,\alpha) $ such that, for every profile $P\in \cP$, the {degree of~$ P $ in~$ (T,\alpha) $} is $\delta_e(P)$ and the maximal degree of $T$ is at most $\max\{\delta_e(\cP),3\}$.
\end{THM}
\begin{proof}
 Let $\cF_e^s$ be the subset of $\cF_e$ consisting of, for every profile $P$, all stars from $\cF_e$ of size $\delta_e(P)$ contained in $P$, together with all stars of size at most $\max\{\delta_e(\cP),3\}$ from $\cF_e$ not contained in any profile. For any star $\sigma$ and any shift $\sigma_\vs^\vr$ of $\sigma$ we have $\abs{\sigma}\ge \abs{\sigma_\vs^\vr}$. Further, $S_k$ is $\cF_e$-separable by \cref{cor:efficient_closed}. Moreover, the shift of a star cannot contain any profile which does not contain the original star by \cref{lem:shift}, thus $S_k$ is also $\cF_e^s$-separable.
 
 Thus, all we need to show is that applying \cref{thm:duality} cannot result in an $\cF_e^s$-tangle, the rest of the proof can then be carried out as the proof of \cref{thm:efficient}:
 Instead of $S$-trees over $\cF_e$ we now consider $S$-trees over $\cF_e^s$, and observe that the shifting argument in the proof of \cref{thm:efficient} again shifts stars in $\cF_e^s$ to stars in $\cF_e^s$.
 
  However, applying \cref{thm:duality} indeed cannot result in an $\cF_e^s$-tangle: Such a tangle cannot be a regular profile, since by our definition of $\delta_e(P)$, there is a star in $\cF_e^s$ contained in $P$. But every consistent orientation which is not a regular profile either contains a star $\{\vs\}$ for a cosmall separation $\vs$ -- each such star is also contained in $\cF_e$ -- or contains, by \cref{lem:profiles_stars} a star of size $3$ not contained in any profile. Either such star is also contained in $\cF_e^s$ by definition.
\end{proof}

\section{Tangles of mixed orders}\label{sec:different_order}
In this section we would like to use the ideas from \cref{sec:efficiency} to obtain a proof of \cref{thm:tot_hundermark} using tangle-tree duality. The challenge of \cref{thm:tot_hundermark} compared to \cref{thm:efficient} is that the set of profiles $\cP$ considered in \cref{thm:tot_hundermark} consists of profiles of different orders. In particular, there might be profiles $P_1$ and $P_2$ in $\cP$ which are efficiently distinguished by separations of order $k$, say, and there might be another profile $Q\in \cP$ which has only order $l<k$ and thus does not orient the separations which efficiently distinguish $P_1$ and $P_2$. Thus, we cannot simply require the stars in our set $\cF$ to be contained in at most one profile: the resulting $S$-tree over $\cF$ would not necessarily distinguish all profiles in $\cP$, for example it might not distinguish the profiles $P_1$ and $Q$ from above. Our solution to this problem will be to restrict the set of stars further by additionally requiring that all the separations in a star in $\cF$ `could be oriented' by every profile in $\cP$, even if that profile has lower order than the separation considered.

With this further restricted set of stars however $S$ will no longer be $\cF$-separable, but it will only fail to do so under rather specific circumstances. Thus in order to obtain a result in the fashion of \cref{thm:tot_hundermark}, we shall first proof a slightly stronger version of \cref{thm:duality}, which allows us to exclude this specific situation im the requirement of $\cF$-separability. The proof of this stronger version of \cref{thm:duality} is different from the original one of \cref{thm:duality} in \cite{TangleTreeAbstract}: our proof of \cref{thm:modifiedTTD} also serves as an alternative proof of \cref{thm:duality} which is slightly shorter than the original one and perhaps neater and less technical.

\subsection{A Short Adventure into Duality}\label{subsec:duality}
As mentioned above, we will prove the following slight strengthening of \cref{thm:duality}:

\begin{THM}\label{thm:modifiedTTD}
    Let $ U $ be a finite universe, $ S\sub U $ a separation system, and $ \cF\sub 2^{\vS} $ a set of stars such that $ \cF $ is standard for $ S $ and $ S $ is critically~$ \cF $-separable. Then precisely one of the following holds:
    \begin{itemize}
        \item there is an $ S $-tree over $ \cF $;
        \item there is an $ \cF $-tangle of $ S $.
    \end{itemize}
\end{THM}

This theorem is a strengthening in the sense that, 
we weaken the technical assumption that~$ S $ be~$ \cF $-separable to only require~$ \cF $-separability for those separations whose inverse lies in no star of~$ \cF $, rather than for all separations in~$ \vS $. Formally:

A separation~$ \vr $ in~$ \vS $ is~\emph{$ \cF $-critical} if~$ \vr\in\sigma $ for some~$ \sigma\in\cF $, but there is no~$ \sigma'\in\cF $ with~$ {\sigma'\cap r=\menge{\rv}} $. Observe that if~$ \vr\in\vS $ is~$ \cF $-critical then~$ \vr $ is nondegenerate and not forced by~$ \cF $, and in particular~$ \vr $ is nontrivial in~$ S $ since~$ \cF $ is standard for~$ S $. We say that~$ S $ is~\emph{critically~$ \cF $-separable} if for all~$ \cF $-critical~$ \vr,\rv'\in\vS $ with~$ \vr\le\vr' $ there exists an~$ s_0\in S $ with an orientation~$ \vs_0 $ that emulates~$ \vr $ in~$ \vS $ for~$ \cF $ and such that~$ \sv_0 $ emulates~$ \rv' $ in~$ \vS $ for~$ \cF $. Clearly, if~$ S $ is~$ \cF $-separable, then~$ S $ is critically~$ \cF $-separable.

The core argument of our proof of \cref{thm:modifiedTTD} is the following: if there is no~$ S $-tree over~$ \cF $, then every `attempt' at such an~$ S $-tree must fail. Thus, if one starts with some star~$ \sigma $ in~$ \cF $ as the basis for such an~$ S $-tree, and then `glues' for each~$ \vs\in\sigma $ with~$ \menge{\sv}\notin\cF $ some star from~$ \cF $ onto~$ \vs $ that contains~$ \sv $, then one must at some point be unable to find such a star. The resulting attempt is then an~$ S $-tree that is `over~$ \cF $' only for all internal vertices, but not necessarily at the leaves. The fundamental strategy of the proof presented here is to collect the set of all these leaf-separations at which the~$ S $-tree attempts get stuck, and then turn this set into the basis of an~$ \cF $-tangle.

The strategy described here is already present in Mazoit's proof~(\cite{MazoitTWDuality}) of the classical duality theorem for brambles and tree-width in graphs. In~\cite{DiestelBook16noEE} Diestel gives a proof of this graph-theoretic duality theorem that is derived from his and Oum's original proof of~\cref{thm:duality}, applied to the specific~$ \cF $ corresponding to tree decompositions of a certain width. Curiously Mazoit's and Diestel's graph-theoretic proofs are quite similar. One could therefore argue that the strategy for the proofs here comes from re-translating Diestel's translation of~\cref{thm:duality} to this specific graph application back into abstract separation systems.

Let us now make this sketch of a proof formal:

\begin{proof}[Proof of~\cref{thm:modifiedTTD}.]
Let~$ U $ be a finite universe,~$ S\sub U $ a separation system, and~$ \cF\sub 2^{\vS} $ a set of stars such that~$ \cF $ is standard for~$ S $ and that $ S $ is critically~$ \cF $-separable.. We may assume that~$ \es\notin\cF $.

We shall need the following definitions. An {\em $S$-tree attempt} is an~$S$-tree~$(T,\alpha)$ with at least one edge and~$ \alpha(t)\in\cF $ for every internal node~$ t $ of~$ T $. For a leaf~$ t $ of an~$ S $-tree attempt~$ (T,\alpha) $ the~\emph{incoming label} of~$ t $ is the separation~$ \vr $ for which~$ \menge{\vr} $ is associated with~$ t $ in~$ (T,\alpha) $; the~\emph{outgoing label} of~$ t $ is then~$ \rv $. We call such an~$ \vr $ a~\emph{petal} of~$ (T,\alpha) $ if~$ \menge{\vr}\notin\cF $.

We will prove the following assertion which is equivalent to the tangle-tree duality theorem:

\[ 
	\textit{$S$ has an~$\cF$-tangle if and only if every~$S$-tree attempt has a petal.}\tag{$\ast$}\label{prop:petals}
\]

Since the~$ S $-tree attempts without petals are exactly the $ S $-trees that are over~$ \cF $,~\eqref{prop:petals} immediately implies~\cref{thm:duality}.

For our proof of~\eqref{prop:petals} we will use~\cref{lem:TreeSets_6.2,lem:TreeSets_6.3,lem:TreeSets_6.4,lem:TreeSets_6.5}. These lemmas, roughly speaking, say that an~$ S $-tree over a set of stars may be assumed to be `cleaned up', i.e. tight and irredundant; that such a cleaned up~$ S $-tree~$ (T,\alpha) $ is order-respecting; and that then each nontrivial separation can only appear once as a label. In short, cleaning up a given~$ S $-tree over stars enables us to apply~\cref{lem:shift_S-tree} to it. 

	The `only if' direction of~\eqref{prop:petals} is clear, so let us show the backward direction.
	
	Let~$\vL\sub\vS $ be the set of all petals of~$ S $-tree attempts. Suppose first that we can find~$P\sub\vL$ such that~$P$ is a consistent and antisymmetric set that contains at least one petal of every~$S$-tree attempt. Then~$ P $ is an~$ \cF $-tangle of~$ S $: Indeed,~$ P $ orients each separation~$ s $ in~$ S $ since it contains a petal of the~$ S $-tree attempt that is just a single edge labelled with~$ s $; and~$ P $ avoids each star~$ \sigma $ in~$ \cF $ since it is antisymmetric and contains a petal of the~$ S $-tree attempt consisting of one internal node for~$ \sigma $ and one leaf for every element of~$ \sigma $.
	
	So let us show that we can find such a set~$P\sub\vL$. For this pick a~$P\sub\vL$ that is minimal with respect to inclusion subject to the conditions that~$ P $ contains at least one petal of every~$S$-tree attempt and is down-closed in~$ \vL $, that is, such that~$ \vp\in P $ for all~$ \vp\in\vL $ with $ \vp\le\vq $ for some~$ \vq\in P $. Such a~$P$ exists since~$\vL$ itself is a candidate. We claim that this~$P$ is antisymmetric and consistent.
	
	So suppose that~$ P $ is not antisymmetric, or not consistent. Then there are~$ \vr\ne\vs $ in~$ P $ with~$ \rv\le\vs $. In particular we can take~$ \vr $ and~$ \vs $ to be maximal elements of~$ P $. Neither~$ \vr $ nor~$ \vs $ can be co-trivial in~$ S $ since~$ \cF $ is standard for~$ S $ and~$ \vr $ and~$ \vs $ are petals. Therefore neither of the two can be trivial or degenerate either, since this would imply that the other one is co-trivial.
	
	By picking~$ \vr $ and~$ \vs $ among the maximal elements of~$ P $ we ensure that both~$ P\sm\menge{\vr} $ and~$ P\sm\menge{\vs} $ are still down-closed in~$ \vL $. Thus, by the minimality of~$ P $, there are~$ S $-tree attempts~$ (T_r,\alpha_r) $ and~$ (T_s,\alpha_s) $ whose only petals that lie in~$ P $ are~$ \vr $ and~$ \vs $, respectively. We may assume~$ (T_r,\alpha_r) $ and~$ (T_s,\alpha_s) $ to be tight and irredundant by \cref{lem:TreeSets_6.2,lem:TreeSets_6.5}, which implies that~$ \vr $ and~$ \vs $ are the incoming label of exactly one leaf of~$ T_r $ and~$ T_s $, respectively, by \cref{lem:TreeSets_6.4}.
	
	We claim that~$ \rv $ is~$ \cF $-critical. To see this, let~$ v_r $ be the leaf of~$ T_r $ whose incoming edge is labelled with~$ \vr $, and let~$ w_r $ be the neighbour of~$ v_r $ in~$ T_r $. Then~$ w_r $ has an incoming edge labelled with~$ \rv $, and we must have~$ \alpha_r(w_r)\in\cF $, witnessing that~$ \rv $ lies in some star in~$ \cF $: for if~$ \alpha_r(w_r)\notin\cF $ then~$ w_r $ would be a leaf of~$ T_r $ and~$ \rv $ a petal of~$ (T_r,\alpha_r) $. By~$ \rv\le\vs $ and~$ P $ being down-closed in~$ \vL $ we would then have~$ \rv\in P $, contrary to the assumption that~$ \vr $ is the only petal of~$ (T_r,\alpha_r) $ which lies in~$ P $. Suppose now that~$ \sigma\cap \menge{\vr,\rv}=\menge{\vr} $ for some~$ \sigma\in\cF $. Then we can extend~$ (T_r,\alpha_r) $ by~$ \sigma $ at the leaf at which~$ \vr $ appears. This extension of~$ (T_r,\alpha_r) $ is then an~$ S $-tree attempt, of which~$ P $ must contain a petal. Since~$ \vr\notin\sigma $ this petal would be strictly larger than~$ \vr $, contradicting the maximality of~$ \vr $ in~$ P $.
	
	A similar argument shows that~$ \sv $ is~$ \cF $-critical. By the assumption that~$ S $ is critically~$ \cF $-separable we thus find an~$ s_0\in S $ with an orientation~$ \vs_0 $ that emulates~$ \rv $ in~$ \vS $ for~$ \cF $ and such that~$ \sv_0 $ emulates~$ \sv $ in~$ \vS $ for~$ \cF $. Let~$ v_r $ and~$ v_s $ be the leaves of~$ T_r $ and~$ T_s $ with incoming labels~$ \vr $ and~$ \vs $, respectively. Set~$ \alpha_r'\coloneqq(\alpha_r)_{v_r,\vs_0} $ and~$ \alpha_s'\coloneqq(\alpha_s)_{v_s,\sv_0} $. Then~\cref{lem:shift_S-tree} says that~$ (T_r,\alpha_r') $ and~$ (T_s,\alpha_s') $ are~$ S $-tree attempts in which~$ v_r $ is the unique leaf of~$ T_r $ with incoming label~$ \sv_0 $, and that~$ v_s $ is the unique leaf of~$ T_s $ with incoming label~$ \vs_0 $. Let~$ (T,\alpha) $ be the~$ S $-tree obtained from~$ (T_r,\alpha_r') $ and~$ (T_s,\alpha_s') $ by identifying~$ v_r\in T_r $ with the neighbour of~$ v_s $ in~$ T_s $ and vice-versa, and extending the maps~$ \alpha_r' $ and~$ \alpha_s' $ accordingly.
	
	This~$ (T,\alpha) $, too, is an~$ S $-tree attempt since every internal node of $(T,\alpha)$ corresponds to an internal node of ~$ (T_r,\alpha_r') $ or~$ (T_s,\alpha_s') $.
	
	We claim that~$ P $, contrary to its definition, contains no petal of~$ (T,\alpha) $. To see this we consider the leaves of $(T,\alpha)$ and note that by definition every leaf of~$ T $ corresponds to either a leaf of~$ T_r $ other than~$ v_r $, or to a leaf of~$ T_s $ other than~$ v_s $. Thus, let us first consider some leaf $t$ of~$ T_r $ other than~$ v_r $, and let~$ \vp $ be the incoming label of~$ t $ in~$ (T_r,\alpha_r) $. Then~$ {\rv\le\vp} $ by \cref{lem:TreeSets_6.3}, and thus the incoming label of~$ t $ in~$ (T,\alpha) $ is~$ \vs_0\join\vp $. If~$ \alpha_r(t)=\menge{\vp}\in\cF $ then~$ \alpha(t)=\menge{\vs_0\join\vp}\in\cF $ since~$ \vs_0 $ emulates~$ \rv $ in~$ \vS $ for~$ \cF $. Consequently, if~$ \vs_0\join\vp $ is a petal of~$ (T,\alpha) $, then~$ \vp $ is a petal of~$ (T_r,\alpha_r) $. In particular, since~$ \vp\ne\vr $ and~$ \vr $ is the only petal of~$ P $ from~$ (T_r,\alpha_r) $, and~$ P $ is down-closed in~$ \vL $, we know that~$ (\vs_0\join\vp)\ge\vp $ cannot be a petal of~$ (T,\alpha) $ that lies in~$ P $.
	
	By the same argument those leaves of~$ T $ which are leaves of~$ T_s $ other than~$ v_s $ cannot give rise to a petal of~$ (T,\alpha) $ in~$ P $, either. Hence~$ P $ contains no petal from~$ (T,\alpha) $, causing a contradiction. This finishes the proof that~$ P $ is consistent and antisymmetric and hence an~$ \cF $-tangle of~$ S $.
\end{proof}

\subsection{Obtaining a tree-of-tangles theorem for different order tangles from tangle-tree duality}\label{subsec:different_order}

\Cref{thm:modifiedTTD} now allows us to use our methods from \cref{thm:efficient} to prove a tree-of-tangles theorem for different order tangles. More specifically we will obtain a result similar to \cref{thm:tot_hundermark}, however
our construction only works in distributive universes -- that is, $\vr \join (\vs \meet \vt) = (\vr \join \vs) \meet (\vr \join \vs)$, always -- since we need the following result from \cite{EberenzMaster}, which can be also found in \cite{ProfilesNew}:

\begin{LEM}[\cite{ProfilesNew}*{Theorem 3.11}, \cite{EberenzMaster}*{Theorem 1}, strong profile property]\label{lem:distributive} Let $\vU$ be a distributive universe and $\vS \subseteq \vU$ structurally submodular, then
for any profile $P$ of $S$ and any $\vr$ and $\vs\in P$ there does not exists any $\vt\in P$ such that $\rv\join \sv\le \vt$.
\end{LEM}

Moreover, our method will not allow us to distinguish all robust profiles, instead we need a slight strengthening of robustness:
We say that a $k$-profile $P$ is \emph{strongly robust}, if for any $\vs\in P$ and $\vr\in \vU$ where $\vs\join \vr$ and $\vs\join \rv$ both have at most the order of $\vs$ one of $\vs\join \vr$ and $\vs\join \rv$ is in $P$. Note that most instances of tangles, for example tangles in graphs, are strongly robust profiles.

For this section let $\vU$ be a distributive submodular universe and let $\cP$ be some set of pairwise distinguishable strongly robust profiles in $\vU$ (possibly of different order).

To handle the issue, that not all separations in a tree-of-tangles for profiles of different orders are oriented by all the considered profiles, we introduce the following additional definition: A consistent orientation $O$ of $S_k$ \emph{weakly orients a separation $s$ as $\vs$} if $O$ contains a separation $\vr$ such that $\vs\le \vr$. If we want to omit $s$ we just say $O$ \emph{weakly contains $\vs$}.

We will now only consider stars of separations where every separation is at least weakly oriented by all the profiles in $\cP$.
Specifically, we work with the set $\cF_d$ consisting of all stars $\sigma$ with the following properties:
\begin{enumerate}
 \item There exists at most one profile $P\in \cP$ such that $\sigma\subseteq P$.
 \item For every profile $P\in \cP$ such that $\sigma\not\subseteq P$ there exists $\vs\in \sigma$ such that $P$ weakly orients $s$ as $\sv$.
 \item If there exists a $P\in \cP$ such that $\sigma\subseteq P$, then $\sigma$ satisfies \cref{prop:prof_eff}.\label{shift:weakly_orients}
\end{enumerate}

We want to show that $U$ is critically $\cF_d$-separable, and our first step to do so is to show that splices -- which we want to use in separability -- are weakly oriented by every profile in $\cP$.
\begin{LEM}\label{lem:critical_seps}
Let $\vU$ be a distributive submodular universe and let $\cP$ be a set of strongly robust profiles in $\vU$.
Suppose that $\vr$ and $\vs$ are $\cF_d$-critical separations in $\vU$ with $\vr \le \sv$, then
every splice between $\vr$ and $\sv$ is weakly oriented by every profile in $\cP$.
\end{LEM}
\begin{proof}
Since $\vr$ and $\vs$ are $\cF_d$-critical, they are contained in some star in $\cF_d$ and hence weakly oriented by every profile in $\cP$.

Let $t$ be a splice between $\vr$ and $\sv$.
If $t$ is not weakly oriented by every profile in $\cP$, then $\cP$ contains a profile $P$ of order at most $|t|$ which weakly orients $r$ as $\vr$ and $s$ as $\vs$, since every witnessing separation that a profile weakly orients $r$ as $\rv$ or $s$ as $\sv$ also witnesses that it weakly orients $t$.
Let $M_r^P$ be the set of all separations $\vw_r$ in $P$ satisfying $\vr\le \vw_r$ and having minimal possible order with that property.
Let $\vw_r\in M_r^P$ be chosen $\le$-maximally.
Let $\vw_s$ be defined for $s$, accordingly.

Observe that if $\vw_r\le \sv$, respectively, then, by the order-minimality of $M_r^P$, the order of $\vw_r$ is at least $|t|$ so $P$ orients $t$, which contradicts the assumption that $P$ does not weakly orient $t$.
Similarly, $\vw_s\le \rv$ results in a contradiction.

Suppose now that $\vw_r$ crosses $\vs$. 

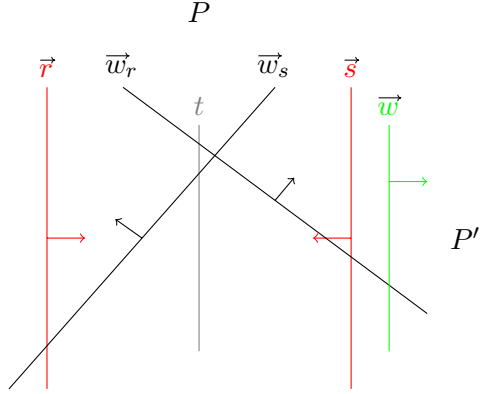
\begin{figure}[h!]
  \centering
  \begin{tikzpicture}
      \draw[red] (-2, -2) -- (-2, 2) node[above] {$\vr$} node[pos=0.5](mid) {} (mid.center)[->] -- +(.5, 0);
      \draw[red] (2, -2) -- (2, 2) node[above] {$\vs$} node[pos=0.5](mid) {} (mid.center)[->] -- +(-.5, 0);
      \draw[gray] (0, -1.5) -- (0, 1.5) node[above] {$t$};
      \draw[black] (3, -1) -- (-1, 2) node[above] {$\vw_r$} node[pos=0.5](mid) {} (mid.center)[->] -- +(.25, .3);
      \draw[green] (2.5, -1.5) -- (2.5, 1.5) node[above] {$\vw$} node[pos=0.75](mid) {} (mid.center)[->] -- +(.5, 0);
      \draw (3.5, 0) node {$P'$};
      \draw (0, 3) node {$P$};
      \draw[black] (-2.5, -2) -- (1, 2) node[above] {$\vw_s$} node[pos=0.5](mid) {} (mid.center)[->] -- +(-.35, .25);
  \end{tikzpicture}
\end{figure}

We claim that every profile $P'$ in $\cP$ which weakly orients $s$ as $\sv$ also weakly contains either $\sv\vee\wv_r$ or $\sv \vee\vw_r$.
This then impies that $\{\sv,\vs\wedge \vw_r,\vs\wedge \wv_r\}$ is a star in $\cF_d$, which will contradict the $\cF_d$-criticality of $\vs$.

So suppose that $P'$ weakly orients $s$ as $\sv$, witnessed by some $\vw \in P' $ with $ \vw\ge \sv$.

If $\vw_r\vee \vw$ had order at most the order of $\vw_r$, this would contradict the choice of $\vw_r$: By \cref{lem:distributive} applied to the separations $\vw_r,\vw_s, \wv\wedge\wv_r $, the profile $P$ would need to contain $\vw_r\vee \vw$ which contradicts the choice of $\vw_r$ being $\le$-maximal in $M_r^P$.

Similarly, if $\wv\wedge\vw_r $ had order less than the order of $\vw_r$, this would contradict the choice of $\vw_r$: By consistency $P$ would need to contain $\wv \wedge \vw_r $  which contradicts the definition of $M_r^P$, from which $\vw_r$ was chosen.

Thus, by submodularity, $\vw\wedge \vw_r$ has order less than the order of $w$, and $\vw\wedge \wv_r$ has order at most the order of $w$.
Hence, as $P'$ is strongly robust, $P'$ contains either $\wv\vee\wv_r$ or 
$\wv\vee \vw_r$ and therefore either weakly orients $\vs\wedge \vw_r$ as $\sv\vee\wv_r$  or $\vs\wedge \wv_r$ as $\sv \vee\vw_r$.

This proves the claim which results in a contradiction to the assumption that $\vs$ is $\cF_d$ critical.
Thus we may suppose that $\vw_r$ does not cross $\vs$ and, by a symmetric argument, that $\vw_s$ does not cross $\vr$. 
Hence $\vr\le \vw_s$ and $\vs\le \vw_r$. We may therefore assume without loss of generality that $\vw_r=\vw_s$. 

\begin{figure}[h]
  \centering
  \begin{tikzpicture}
      \draw[red] (-2, -2) -- (-2, 2) node[above] {$\vr$} node[pos=0.5](mid) {} (mid.center)[->] -- +(.5, 0);
      \draw[red] (2, -2) -- (2, 2) node[above] {$\sv$} node[pos=0.5](mid) {} (mid.center)[->] -- +(.5, 0);
      \draw[gray] (0, -1.5) -- (0, 1.5) node[above] {$t$};
      \draw[black, rounded corners=.5cm] (-1.5, 2.5) -- (0,0.5) -- (1.5, 2.5) node[above] {$\vw_r = \vw_s$} node[pos=0.7](mid) {} (mid.center)[->] -- +(-.3, .25);
      \draw (0, 3) node {$P$};
  \end{tikzpicture}
\end{figure}

If $\vw_r=\vw_s$ crosses $t$ then, by the choice of $t$, that neither $\vw_r\wedge \vt$ nor $\vw_r\wedge \tv$ has order less than $|t|$, thus $\vw_r\vee \vt$ and $\vw_r\vee\tv$ both have order at most the order of $\vw_r$. By the strong robustness of $P$ applied to $\vw_r, \wv_r\wedge \tv$ and $\wv_r\wedge\vt$, we know that either $\vw_r\vee \vt\in P$ or $\vw_r\vee\tv\in P$. However, both contradict the $\le$-maximal choice of $\vw_r$.
So, instead $\vw_r$ is nested with $t$, that is, $t$ has an orientation $\vt$ such that $\vt\le \vw_r$, so $t$ is weakly oriented by $P$, as claimed.
\end{proof}

Note that the assumption that our profiles are \emph{strongly} robust is essential in this argument, for example for the case $\vw_r=\vw_s$: If we only assume robustness, we can not conclude that $P$ contains either $\vw_r\join \vt$ or $\vw_r\join \tv$ and thus would not obtain a contradiction. 

The next step is to verify that shifting with a splice as in \cref{lem:critical_seps} maps stars in $\cF_d$ to stars in $\cF_d$, which will prove that $U$ is critically $\cF_d$-seperable:

\begin{LEM}\label{lem:critical_shift}
Let $r$ and $s_0$ be separations which are weakly oriented by every profile in $\cP$ and suppose that
$\vs_0$ us a splice for $\vr$. Let $\sigma\in \cF_d$ be a star which contains a separation $\vx\ge \vr$. Then the shift $\sigma_\vx^\vs_0$
of $\sigma$ from $\vx$ to $\vs_0$ is again an element of $\cF_d$.
\end{LEM}
\begin{proof}
Since $\vs_0$ is a splice for $\vr$, by \cref{lem:splice_shift}, $\vs\join \vs_0$ has at most the order of $\vs$ for every $\vs \ge \vr$.

Let $\sigma$ be any star in $\cF_d$ containing a separation $\vx\ge \vr$.
By the above, if $\sigma \subseteq \vS_k$ for some $k$ then also the shift $\sigma_\vx^\vs_0$ is a subset of $\vS_k$.
Hence by \cref{lem:shift}, every profile in $U$ which contains $\sigma_\vx^\vs_0$ also contains $\sigma$. Now if some profile $P$ contains $\sigma$, then $P$ orients every separation in $\sigma_\vx^\vs_0$, and thus either $P$ contains the inverse of some separation in $\sigma_\vx^\vs_0$ or $\sigma_\vx^\vs_0\subseteq P$.

Hence, by \cref{lem:efficient_shift} it is enough to show that every profile from $\cP$ which, for some $\vy\in \sigma$, weakly contains $\yv$ also weakly contains $\yv'$ for some separation $\vy'\in \sigma_\vx^\vs_0$.

So suppose such a profile $P$, for some $\vy\in \sigma$, weakly contains $\yv$ and suppose that this is witnessed by $\wv_y\in P$.
If $\vr\le \vy$, then $\yv$ is shifted onto $\yv\meet \sv_0$ and therefore $\vw_y$ also witnesses that $P$ weakly contains $\yv\meet \sv_0$ while $\vy\join \vs_0\in \sigma_\vx^\vs_0$. Thus we may suppose that $\vr\le \yv$ and therefore that $\yv$ is shifted onto $\yv\join\vs_0$.

If $P$ weakly orients $s_0$ as $\sv_0$, then $P$ also weakly contains $\vy\meet \sv_0\le \sv_0$ while $\yv\join \vs_0\in \sigma_\vx^\vs_0$.

Thus we may suppose that $P$ weakly orients $s_0$ as $\vs_0$, witnessed by $\vw_0\in P$. 

By our assumptions on $s_0$ we know that the order of $\vs_0\meet\wv_y$ is at least the order of $s_0$ and thus, by submodularity, $\sv_0\wedge\vw_y$ has order at most the order of $w_y$, i.e., it is oriented by $P$. By \cref{lem:distributive} applied to $\vw_0, \wv_y\in P$ and $\sv_0\wedge\vw_y$ we can therefore conclude that $P$ contains $\vs_0\vee\wv_y$, i.e., $P$ weakly contains $\yv\vee\vs_0\le \vs_0\vee\wv_y$.
\end{proof}

In order to use our stronger tangle-tree duality theorem \cref{thm:modifiedTTD} with our set $\cF_d$ of stars to obtain a tree of tangles for strongly robust profiles it only remains for us to show that this application cannot result in an $\cF_d$-tangle. We do this in the following two lemmas.

\begin{LEM}\label{lem:efficient_P_star}
    For every profile $P$ in $\vU$ and every set $\cP'$ of strongly robust profiles in $\vU$ distinguishable from $P$, there exists a nested set $N$ which distinguishes $P$ efficiently from all the profiles in $\cP'$.
\end{LEM}
\begin{proof}
 For every profile $Q\in \cP'$ pick a $\le$-minimal separation $\vs_Q\in P$ which efficiently distinguishes $Q$ from $P$.
 We claim that the set $N$ consisting of all these separations $s_Q$ is nested and therefore as claimed.
 
 So suppose that this is not the case, so $\vs_Q$ and $\vs_{Q'}$, say, cross. We may assume without loss of generality that $\abs{\vs_Q}\le \abs{\vs_{Q'}}$. 
 Now $\vs_Q\join \vs_{Q'}$ has order at least the order of $\vs_{Q'}$ since otherwise, by the profile property, $\vs_Q\join \vs_{Q'}$ would also distinguish $P$ and $Q'$ and would thus contradict the fact that $\vs_{Q'}$ did so efficiently. 
 Thus $\abs{\vs_Q\meet \vs_{Q'}}\le \abs{\vs_Q}$.
 
Now $Q'$ orients $\vs_Q$ and it cannot contain $\sv_Q$ since then, by the profile property, $\vs_Q\meet \vs_{Q'}$ would also distinguish $P$ and $Q'$ efficiently and would therefore contradict the $\le$-minimal choice of $\vs_{Q'}$. 

Thus $\vs_Q\in Q'$. 
Now $\abs{\sv_Q\meet \vs_{Q'}}>\abs{\vs_{Q'}}$ since otherwise, again by the profile property, $\sv_Q\meet \vs_{Q'}$ contradicts the $\le$-minimal choice of $\vs_{Q'}$.

Thus, by submodularity, $\abs{\vs_Q\meet \sv_{Q'}}<\abs{\vs_Q}$ and $\abs{\vs_Q\meet \vs_{Q'}}\le \abs{\vs_Q}$. 
But, by strong robustness, either $\sv_Q\join \vs_{Q'}$ or $\sv_Q\join \sv_{Q'}$ is in $Q$. 
In particular, $\vs_Q\meet \sv_{Q'}$ or $\vs_Q\meet \vs_{Q'}$ efficiently distinguishes $P$ and $Q$ and therefore contradicts the $\le$-minimal choice of $\vs_Q$.
\end{proof}

Unlike for structurally submodular separation systems in \cref{lem:non_profiles} or efficient distinguishers in \cref{lem:profiles_stars}, in this setup we can not necessarily find a star in $\cF_d$ which is contained in $O$ but in no profile in $\cP$ for every orientation $O$ of $U$ which not include any profile in our set $\cP$ of strongly robust profiles .
This is because we require that every profile in $\cP$ weakly orients a separation in our star outwards, but the stars constructed in \cref{lem:profiles_stars}, for example, do not necessarily have this property. Thus we are going to, instead, find a star $\sigma$ contained in both $O$ and exactly one profile from $\cP$. Since each such star also lies in $\cF_d$, this will be enough to ensure that our application of \cref{thm:modifiedTTD} does not result in an $\cF_d$-tangle.

\begin{LEM}\label{lem:no_tangle}
    For every consistent orientation $O$ of $\vU$ and every set $\cP\neq \emptyset$ of distinguishable strongly robust profiles in $\vU$ there exists a star $\sigma$ in $\cF_d$ 
    contained in $O$.
\end{LEM}
\begin{proof}
Pick a star $\sigma$ (not necessarily from $\cF_d$) with the following properties:
\begin{enumerate}[(i)]
\item $\sigma\subseteq O$.
\item $\sigma$ is contained in at least one profile in $\cP$.
\item \Cref{prop:prof_eff} is satisfied for every profile $P\in \cP$ such that $\sigma\subseteq P$.
\item Every $P\in \cP$ either contains $\sigma$ or weakly contains $\sv$ for some separation $\vs\in \sigma$.
\item For every separation $\vs\in \sigma$ and any profile $\sigma\subseteq P$ there exists a profile $Q\in\cP$ such that $s$ is a efficient $P$-$Q$ distinguisher.
\end{enumerate}
Note that the empty set is such a star. Let us further assume that we choose our star $\sigma$ fulfilling (i)-(v) such that as few profiles in $\cP$ as possible contain $\sigma$.

If only one profile contains $\sigma$ then $\sigma\in \cF_d$ is as desired, so let us suppose for a contradiction that there are at least two such profiles.

Pick two such profiles $P_1,P_2\supseteq \sigma$ such that the order of an efficient $P_1$-$P_2$-distinguisher is as small as possible. Pick an efficient $P_1$-$P_2$-distinguisher $s$ which crosses as few elements of $\sigma$ as possible. $O$ orients $s$, say $\vs\in O$. If $s$ is nested with $\sigma$, the maximal elements of $\sigma\cup \{\vs\}$ form a star violating the definition of $\sigma$: Every profile containing this new star also contains $\sigma$. To see that (iii) is fulfilled, note that there is no profile $P\supseteq \sigma$ in $\cP$ such that $\vs\in P$ for which there is a $\vs'$ of lower order than $\vs$ such that $\vs\le\vs'\in P$, since such an $\vs'$ would be a distinguisher of lower order than $\vs$ for some pair of profiles containing $\sigma$, contrary to the coice of $s$.

Thus we may assume that $s$ is not nested with $\sigma$, say $s$ crosses $\vt\in \sigma$. Since, by(v), there is some profile $Q\ni\tv$ for which $t$ is an efficient $P_1$-$Q$-distinguisher, we know that at least one of $\vs\wedge\vt$ and $\sv\wedge\vt$ has order at least the order of $t$: Otherwise this would contradict the fact that $t$ is an efficient $P_1$-$Q$-distinguisher by robustness (if $|t|<|s|$) or the profile property (if $|s|\le|t|$) of $Q$.

Thus by submodularity the order of at least one of $\vs\vee\vt$ and $\sv\vee\vt$ is at most the order of $s$ and that separation is therefore also an efficient $P_1$-$P_2$-distinguisher (by the profile property and consistency), which would make it a better choice for $s$, a contradiction.

Thus $\sigma$ contains precisely one profile and therefore, by construction, $\sigma\in \cF_d$.
\end{proof}

Together with \cref{thm:modifiedTTD}, these lemmas give a proof of a tree-of-tangles theorem for strongly robust profiles of different orders in a submodular universe. This theorem does not give efficient distinguishers; we will deal with efficiency in a later step.

\begin{THM}[Tree-of-tangles theorem for different orders] \label{thm:different_non}
Let ${\vU=(\vU,\le,^*,\vee,\wedge,|\ |)}$ be a submodular distributive universe of separations. Then for
every distinguishable set $\cP$ of strongly robust profiles in $\vU$ there is a nested set $T = T(\cP)\subseteq U$ of
separations such that:
\begin{enumerate}[(i)]
 \item  every two profiles in $\cP$ are distinguished by some separation in $T$;
 \item for any profile $P\in \cP$, any maximal $\vs\in P\cap \vT$ and any $\vs'\in P$ such that $\vs\le\vs'$ we have $|\vs|\le|\vs'|$. \label{thm_property:efficient}
\end{enumerate}
\end{THM}
\begin{proof}
By \cref{lem:critical_seps} and \cref{lem:critical_shift} the set $U$ is critically $\cF_d$-separable for the set  $\cF_d$ defined above.
Thus we can apply \cref{thm:modifiedTTD}. This can, by \cref{lem:no_tangle}, not result in an $\cF_d$-tangle, thus there is an $U$-tree over $\cF_d$. 
By \cref{lem:TreeSets_6.2} we may assume this $U$-tree to be irredundant.
The set of separations associated to edges of this tree is then a nested set $T$.

Every profile in $\cP$ induces a consistent orientation of $T$, since all the separations in $T$ are weakly oriented by every profile in $\cP$. The maximal elements of this orientation form a star $\sigma_P$ in $\cF_d$, and this star is a subset of $P$ by the definition of $\cF_d$.

To see that $T$ distinguishes every pair of profiles in $\cP$, consider two profiles $P$ and $Q$ in $\cP$.
These two profiles cannot induce the same orientation of $T$, since then $\sigma_P = \sigma_Q$ would be a subset of both $P$ and $Q$, contradicting the definition of $\cF_d$. Thus some $\vs \in \sigma_P$ witnesses that $P$ weakly orients some $\vt \in \sigma_Q$ as $\tv$ and, vice versa, $\vt$ witnesses that $Q$ weakly contains $\sv$. Of these two separations $s$ and $t$, the one of lower order is thus a $P$--$Q$-distinguisher in $T$.

\Cref{thm_property:efficient} is then immediate from the definition of $\cF_d$.
\end{proof}

Note that the nested set constructed in \cref{thm:different_non} does not yet necessarily distinguish any two profiles \emph{efficiently}. However, we can use \cref{thm:efficient} in combination with \cref{thm:different_non} to obtain such a set:
  \begin{THM}[Efficient tree\,-\,of\,-\,tangles theorem for different order profiles] \label{thm:different_eff}
Let\linebreak  ${\vU=(\vU, \le, ^*,\vee,\wedge,|\ |)}$ be a submodular distributive universe of separations. Then for
every distinguishable set $\cP$ of strongly robust profiles in $\vU$ there is a nested set $T = T(\cP)\subseteq U$ of
separations such that every two profiles in $\cP$ are efficiently distinguished by some separation in $T$.
\end{THM}
\begin{proof} Let $k$ be the maximal order of a profile in $U$.
 Let $T$ be the $U$-tree over $\cF_d$ from the proof of \cref{thm:different_non}. We consider the $\subseteq$-~maximal subtrees $T_i$ of $T$ with the property that no internal node of $T_i$ corresponds to a profile in $\cP$. Clearly $T=\bigcup_{i=1}^mT_i$ and no two $T_i$ share an edge.
 
 We are going to simultaneously replace each of the nested sets of separations corresponding to the $T_i$s with other separations in such a way that the resulting set of separations is still nested and we ensured that every pair of profiles contained in some $T_i$ is efficiently distinguished by this new set of separations.

 So, given some $T_i$, let $\cP_i$ be the set of profiles in $\cP$ living, in $T$, in one of the leaves of $T_i$. Let $\vL_i$ be the set of all separations associated to one of the directed edges adjacent and pointing away from such a leaf. Note that $\vL_i$ is a star.
 For every $\vs\in \vL_i$ let $P_s \in \cP_i$ be the unique profile corresponding to a leaf of $T_i$ and containing $\sv$. 
 
 It is easy to check that for any two profiles $P$ and $Q$ in $\cP_i$ there is a efficient $P$--$Q$-distinguisher $t$ which is nested with all of $\vL_i$: Pick one $t$ which is nested with as many separations from $\vL_i$ as possible. Now $t$ cannot cross an $\vs\in \vL_i$ such that $P_s=P$ or $P_s=Q$, as in that case, for $\vt\in P_s$, either $\vt\join \sv$ or $\vt\meet\sv$ would, by submodularity, consistency and the profile property, be an efficient $P$-$Q$--distinguisher and as such contradict the choice of $t$ by \cref{lem:fish}. If on the other hand $t$ crosses some $\vs\in \vL_i$, such that $P_s\notin \{P,Q\}$, then not both of $\sv\join \vt$ and $\sv\join \tv$ can have order less than the order of $s$ by the profile property since, by \cref{thm_property:efficient}, there is no $\sv'\in P_s$ such that $\sv\le \sv'$ and $|\vs|>|\vs'|$.
 Thus the order of either $\vs\join \vt$ or $\vs\join \tv$ is at most the order of $t$, however by \cref{lem:distributive} and the fish \cref{lem:fish} this separation then contradicts the choice of $t$.

 Moreover, there exists such an efficient $P$--$Q$-distinguisher $t$ which has an orientation $\vt$ such that $\vs\le \vt$ for every $\vs \in \vL_i$: Otherwise $\sv\le\vt$ for some orientation of $t$ and if neither $P=P_s$ nor $Q=P_s$ then both $P$ and $Q$ would weakly orient $t$ as $\tv$ since they weakly contain $\vs$. On the other hand if $P=P_s$, say, then, again by \cref{thm_property:efficient}, the order of $t$ is at least the order of $s$, thus $s$ itself would be the required efficient $P$--$Q$-distinguisher.
 
 Now consider, for every $T_i$, the set $U^i$ of all separations $t$ in $U$ nested with $\vL_i$ and fulfilling the additional property of having, for every $\vs\in \vL_i$, an orientation such that $\vs\le \vt$, i.e. $U^i$ is the set of all separations in $U$ inside of $\vL_i$. $\vU^i$ is closed under $\join$ and $\meet$ in $\vU$ by the fish \cref{lem:fish}, thus the restriction of $U$ to $U^i$ is again a submodular universe of separations.
 
Given any $\vs\in \vL_i$, the down-closure of $\sv$ is a regular profile of $U^i$. Note that every efficient distinguisher for the profiles induced by $\sv_1$ and $\sv_2\in \Lv_i$ on $U^i$ is also an efficient distinguisher of $P_{s_1}$ and $P_{s_2}$.

By \cref{thm:efficient} applied to the set of all separations of order less than $k$ in $U^i$, we thus find a $U^i$-tree $\hat{T^i}$ over $\cF_e$ (defined for $\cP_i$). The corresponding nested set $N_i$ efficiently distinguishes all these profiles induced by some $\sv_i\in \Lv_i$. 

But now the nested $N$ given by $\bigcup_{i=1}^m(N_i\cup L_i)$ is as desired: It is easy to see that this set is nested and every $N_i$ efficiently distinguishes any two profiles in $\cP_i$. Moreover, we  only ever changed separations inside of $\vL_i$ for every $T_i$.

The set $N$ also contains an efficient $P$--$Q$-distinguisher for profiles $P$ and $Q$ in different~$T_i$s:
A profile $R$ whose node in $T$ lies on the path between the nodes containing $P$ and $Q$, respectively, also does so in the tree induced by $N$. Thus, if we have efficient distinguishers for $P$ and $R$ and for $R$ and $Q$, respectively, in $N$, then one of the two is also an efficient $P$--$Q$-distinguisher. An inductive application of this argument proves the claim, that
the set $N$ efficiently distinguishes any two profiles in $\cP$.
\end{proof}

\bibliography{collective}
\end{document}